\documentclass{amsart}


\usepackage{amsmath,amssymb, amstext,amsthm,amsfonts}
\usepackage[dvips]{graphicx}

\thanks{Partially supported by CNPq, FAPERJ and PRONEX/DS from Brazil.}

\newcommand{\cl}{\operatorname{Cl}}
\newcommand{\inte}{\operatorname{Int}}
\newcommand{\diff}{\operatorname{Diff}}
\newcommand{\emb}{\operatorname{Emb}}

\newcommand{\sink}{\operatorname{Sink}}
\newcommand{\sou}{\operatorname{Source}}

\newcommand{\de}{\operatorname{Deg}}
\newcommand{\sad}{\operatorname{Saddle}}

\newcommand{\ta}{\operatorname{Tang}}

\newtheorem{theorem}{Theorem}[section] 
\newtheorem{lemma}[theorem]{Lemma}     
\newtheorem{corollary}[theorem]{Corollary}
\newtheorem{proposition}[theorem]{Proposition}
\newtheorem{defi}[theorem]{Definition}

\title[Some properties of surface diffeomorphisms]
 {Some properties of surface diffeomorphisms} 

\author{A. Arbieto, C. A. Morales}

\address{Instituto de Matem\'atica, Universidade Federal do Rio de Janeiro, P. O. Box 68530, 21945-970 Rio
de Janeiro, Brazil.}
\email{arbieto@im.ufrj.br, morales@impa.br}

\begin{document}
\maketitle

\begin{abstract}
We obtain some properties of $C^1$ generic surface diffeomorphisms as
finiteness of {\em non-trivial} attractors, approximation by diffeomorphisms with only a finite number
of {\em hyperbolic} homoclinic classes, equivalence between
essential hyperbolicity and the hyperbolicity of all {\em dissipative} homoclinic classes (and
the finiteness of spiral sinks).
In particular, we obtain the equivalence between finiteness of sinks and finiteness of spiral sinks,
abscence of domination in the set of accumulation points of the sinks, and
the equivalence between Axiom A and the hyperbolicity of all homoclinic classes.
These results improve \cite{A}, \cite{a}, \cite{m} and
settle a conjecture by Abdenur, Bonatti, Crovisier and D\'{i}az \cite{abcd}.
\end{abstract}

\section{Introduction} 
\label{intro}

\noindent
One of the most important open problems about dynamical systems of surfaces is Smale's Conjecture (\cite{smale} p.779),
that asserts the denseness of the diffeomorphisms satisfying Axiom A
among surface diffeomorphisms.
If true, then some relevant properties of Axiom A like
essential hyperbolicity,
finiteness of sinks, non-trivial attractors and homoclinic classes
must be true in the $C^1$ generic world too.

The main aim of this paper is to make some contributions for positive solution of the Smale's Conjecture.
Indeed, we first prove that the number of {\em non-trivial} attractors of a $C^1$ generic surface diffeomorphism is finite.
Also that every surface diffeomorphism can be $C^1$ approximated by ones for which
number of {\em hyperbolic} homoclinic classes is finite.
Yet, a $C^1$ generic surface diffeomorphism is essentially hyperbolic if and only if
its dissipative homoclinic classes are all hyperbolic.
In particular,
a $C^1$ generic orientation-preserving surface diffeomorphism is essentially hyperbolic if and only
the number of {\em spiral} sinks is finite.
In addition, the number of sinks of a $C^1$ generic orientation-preserving surface diffeomorphism is finite if and only
if the corresponding number for spiral sinks is finite too. Moreover,
the set of accumulation points of the sinks of a $C^1$ generic surface diffeomorphism
cannot have any dominated splitting.
Finally, a $C^1$ generic surface diffeomorphism whose homoclinic classes are all hyperbolic satisfies Axiom A.
Our results improve the Araujo's thesis \cite{A}, results by Asaoka\cite{a}, the second author \cite{m} and
solve Conjecture 1 in \cite{abcd} p.130.

\section{Statement of the results}

\noindent
By {\em surface diffeomorphism} we mean a diffeomorphism of class $C^1$ of a compact connected boundaryless two-dimensional Riemannian manifold $M$.
The corresponding space equipped with the $C^1$ topology will be denoted by $\diff^1(M)$.
A subset of $\diff^1(M)$ is {\em residual} if it is a countable intersection of open and dense subsets.
We say that a {\em $C^1$ generic surface diffeomorphism satisfies a certain property P} if
there is a residual subset $\mathcal{R}$ of $\diff^1(M)$ such that P holds for every element of $\mathcal{R}$.
A surface diffeomorphism is {\em orientation-preserving}
if $M$ is orientable and $f$ preserves a given orientation.
The closure operation is denoted by $\cl(\cdot)$.

Given a surface diffeomorphism $f$ and a point $x$ we define the {\em omega-limit set},
$$
\omega(x)=\left\{y\in M : y=\lim_{k\to\infty}f^{n_k}(x)\mbox{ for some integer sequence }n_k\to\infty\right\}.
$$
(when necessary we shall write $\omega_f(x)$ to indicate the dependence on $f$.)
We say that a subset $\Lambda\subset M$ is
{\em invariant} if $f(\Lambda)=\Lambda$;
{\em transitive} if there is $x\in\Lambda$ such that
$\Lambda=\omega(x)$; and {\em non-trivial} if it does not reduces to a periodic orbit.

An {\em attractor} is a transitive set $A$ exhibiting a neighborhood $U$ such that
$$
A=\displaystyle\bigcap_{n\in\mathbb{N}}f^n(U).
$$
With these definitions we can state our first result.

\begin{theorem}
\label{attractor}
The number of {\em non-trivial} attractors of a $C^1$ generic surface diffeomorphism is finite.
\end{theorem}

A compact invariant set $\Lambda$ of $f$ is {\em hyperbolic} if
there are a continuous invariant tangent bundle decomposition
$T_\Lambda M=E_\Lambda^s\oplus E_\Lambda^s$ over $\Lambda$
and positive numbers $K,\lambda$ such that
$$
\|Df^n(x)/E_x\|\leq Ke^{-\lambda n}
\quad\mbox{ and } \quad
\|Df^{-n}(x)/F_{f^n(x)}\|\leq K^{-1}e^{\lambda n}
\quad\forall x\in \Lambda,\forall n\in\mathbb{N}.
$$

Given $f\in\diff^1(M)$ we say that
$x\in M$ is {\em periodic}
if $f^n(x)=x$ for some positive integer $n$. The minimum of such integers is the so-called {\em period} denoted by $n_x$ (or $n_{x,f}$ to emphasize $f$).
The {\em eigenvalues} of a periodic point $x$ will be those of the linear isomorphism
$Df^{n_x}(x):T_xM\to T_xM$.

A periodic point $x$ is a {\em saddle} if
it has eigenvalues of modulus less than and bigger than $1$.
The invariant manifold theory \cite{hps} asserts that  for every saddle $x$
there are invariant manifolds tangent at $x$ to the eigenspaces associated to the
eigenvalues of modulus less and bigger than one respectively.
A {\em homoclinic point} associated to $x$ is a point 
where such manifolds meet. A {\em transverse homoclinic point} is one where these manifolds
meet transversally.
We denote by $H_f(x)$ the closure of the transverse homoclinic points associated to $x$. 
A {\em homoclinic class} is a subset equals to $H_f(x)$ for some saddle $x$.

\begin{theorem}
\label{thAA'}
Every surface diffeomorphism can be $C^1$ approximated by ones for which the
number of {\em hyperbolic} homoclinic classes is finite.
\end{theorem}

A {\em hyperbolic attractor} is an attractor which is simultaneously a hyperbolic set.

The {\em basin} of any subset $\Lambda\subset M$ is defined by
$$
W^s(\Lambda)=\{y\in M : \omega(y)\subset \Lambda\}.
$$
(Sometimes we write $W^s_f(\Lambda)$ to indicate the dependence on $f$).

We say that $f$ is {\em essentially hyperbolic} if
it exhibits finitely many hyperbolic attractors whose basins form an open and dense subset of $M$.

We call a homoclinic class $H$ of a surface diffeomorphism $f$
{\em dissipative} if $H=H_f(x)$ for some $x\in\sad_d(f)$.

We shall prove that, $C^1$ generically, the hyperbolicity of the dissipative homoclinic classes
is equivalent to essentially hyperbolicity.

\begin{theorem}
\label{hyp->ess}
A $C^1$ generic surface diffeomorphism is essentially hyperbolic if and only if
its dissipative homoclinic classes are all hyperbolic.
\end{theorem}

A periodic point $x$ is a {\em sink} (resp. {\em source})
if all its eigenvalues have modulus less (reps. bigger) than $1$. If, additionally, such eigenvalues are
not real, then we say that $x$ is a {\em spiral sink} (resp. {\em spiral source}).

\begin{theorem}
\label{araujo}
A $C^1$ generic orientation-preserving surface diffeomorphism $f$ is essentially hyperbolic if and only
the number of spiral sinks of $f$ is finite.
\end{theorem}

Since the number of sinks of an essentially hyperbolic diffeomorphism is finite, we obtain the following corollary.

\begin{corollary}
\label{thAA}
The number of sinks of a $C^1$ generic orientation-preserving surface diffeomorphism is finite if and only
if the number of spiral sinks is finite too.
\end{corollary}

Another corollary is given as follows.

An {\em non-empty} invariant set $\Lambda$
has a {\em dominated splitting} with respect to $f$
if there are a continuous invariant tangent bundle decomposition
$T_\Lambda M=E_\Lambda\oplus F_\Lambda$ over $\Lambda$ with $E_x\neq 0$ and $F_x\neq 0$ for all $x\in \Lambda$
and positive numbers $K,0<\lambda<1$ such that
$$
\|Df^n(x)/E_x\|\cdot\|Df^{-n}(f^n(x))/F_{f^{-n}(x)}\|\leq K\lambda^{-n},
\quad\quad\forall x\in \Lambda,\forall n\in\mathbb{N}.
$$
(Notice that this definition is symmetric, i.e., $\Lambda$ has a dominated splitting
for $f$ if and only if it does for $f^{-1}$. By this reason we just say
that $\Lambda$ has a dominated splitting without explicit mention for $f$ or $f^{-1}$).
The set of sinks of a surface diffeomorphism $f$ will be denoted by $\sink(f)$.
We also denote by $\sou(f)$ the set of sources of $f$.

The following improves results
by Asaoka \cite{A} and the second author \cite{m}.

\begin{corollary}
\label{c1}
If $f$ is a $C^1$ generic surface diffeomorphism,
then $\cl(\sink(f))\setminus \sink(f)$ cannot have a dominated splitting(\footnote{During the prepararation of this work, we learn that Crovisier,
Pujals and Sambarino announced a similar result in higher dimension when the dominated subbundle is one-dimensional.}).
\end{corollary}

The {\em nonwandering set} of a surface diffeomorphism
$f$ consists of those points $p$ 
such that
for every neighborhood $U$ of $p$ there is $n\in\mathbb{N}^+$ such that
$U\cap f^n(U)\neq\emptyset$.
We say that $f$ {\em satisfies Axiom A} if its nonwandering set is both hyperbolic and
the closure of the periodic points.
It follows from the Smale's spectral decomposition theorem \cite{hk} that
every surface diffeomorphism satisfying Axiom A is essentially hyperbolic
(but not conversely).

Theorem \ref{hyp->ess} together with the so-called {\em Ma\~n\' e's dichotomy} \cite{M}
will be used to prove the aforementioned conjecture in \cite{abcd}.

\begin{corollary}
\label{thA}
A $C^1$ generic surface diffeomorphism whose homoclinic classes are all hyperbolic satisfies Axiom A.
\end{corollary}

This paper is organized as follows.
In Section \ref{sec1.5} we give some general preliminaries.
In Section \ref{sec2} we define a degenerate point for a surface diffeomorphism as a point which
can be turned by small perturbation into a dissipative saddle
with small angle between their invariant subbundles along the orbit.
We shall use the techniques in \cite{PS} to prove that every degenerate point can be turned into a dissipative homoclinic tangency by small perturbations.
Therefore, $C^1$ generic surface diffeomorphisms satisfy that their
degenerate points belong to the closure of the sinks.
Furthermore, we prove that
every surface diffeomorphism has domination outside $\de(f)$.
We also provide a result closely related to the Araujo's thesis \cite{A}.
In Section \ref{sec4} we prove our results.
Some technical results will be proved in Section \ref{sec-final}.

\section{Preliminars}
\label{sec1.5}

\subsection{Non-explosion for sinks}
Recall that a diffeomorphism is {\em Kupka-Smale} if its periodic points are hyperbolic (i.e. without eigenvalues in the unit circle)
and the corresponding invariant manifolds are in general position.
Denote by $\sink_\mathbb{C}(f)$ the set of spiral sinks of a surface diffeomorphism $f$.

\begin{lemma}
\label{lema0}
There is a residual subset of surface diffeomorphisms $\mathcal{R}_0$ with the following property:
Every $f\in \mathcal{R}_0$ is Kupka-Smale, and, for every $x\notin\cl(\sink(f))$ (resp. $x\notin\cl(\sink_\mathbb{C}(f))$), there
are neighborhoods $U_x$ of $x$, $\mathcal{U}_x(f)$ of $f$ and $U$ of $\cl(\sink(f))$ (resp. $\cl(\sink_\mathbb{C}(f))$)
such that
\begin{equation}
\label{pp}
U_x\cap U=\emptyset \quad\mbox{ and }\quad
\cl(\sink(h))\subset U\mbox{ (resp. }\cl(\sink_\mathbb{C}(h))\subset U),
\quad\quad\forall h\in \mathcal{U}_x(f).
\end{equation}
\end{lemma}

\begin{proof}
We only prove the result for $\sink(f)$ (the proof for $\sink_\mathbb{C}(f)$ is similar).
Denote by $2^M_c$ the set of compact subsets of the surface $M$ and
define
$$
S:\diff^1(M)\to 2^M_c
\quad\mbox{ by }\quad
S(f)=\cl(\sink(f)).
$$
It follows easily from the continuous dependence of the eigenvalues
of a periodic point with respect to $f$ that
this map is {\em lower-semicontinuous},
i.e.,
for every $f\in \diff^1(M)$ and every open set $W$ with $S(f)\cap W\neq\emptyset$
there is a neighborhood $\mathcal{P}$ of $f$ such that $S(g)\cap W$ for all $g\in \mathcal{P}$.

From this and well-known properties of lower-semicontinuous maps
\cite{k1}, \cite{k}, we obtain
a residual subset $\mathcal{R}_0\subset \diff^1(M)$ where
$S$ is {\em upper-semicontinuous}, i.e.,
for every $f\in \mathcal{R}_0$ and every compact subset $K$ satisfying $S(f)\cap K=\emptyset$ there is a neighborhood $\mathcal{D}$ of $f$
such that $S(g)\cap K=\emptyset$ for all $g\in\mathcal{D}$.
By the Kupka-Smale theorem \cite{hk} we can also assume that every $f\in \mathcal{R}_0$ is Kupka-Smale.

Now, take $f\in \mathcal{R}_0$.
We already have that $f$ is Kupka-Smale. If $x\notin\cl(\sink(f))$, then there are neighborhoods $U_x$
and $U$ of $x$ and $\cl(\sink(f))$ respectively such that
$U_x\cap U=\emptyset$.
Putting $K=M\setminus U$ in the definition of upper-semicontinuity we obtain a
neighborhood $\mathcal{U}_x(f)$ of $f$ such that
$\cl(\sink(h))\cap (M\setminus U)=\emptyset$ or, equivalently,
$\cl(\sink(h))\subset U$, $\forall h\in\mathcal{U}_x(f)$.
This ends the proof.
\end{proof}

\subsection{Lyapunov stability and neutral sets}
In this subsection we will introduce some results from \cite{cmp}.

Let $\Lambda$ be a compact invariant set of a surface diffeomorphism $f$.
We say that $\Lambda$ is {\em Lyapunov stable for $f$}
if for every neighborhood $U$ of $\Lambda$ there is a neighborhood $V\subset U$ of $\Lambda$ such that
$f^n(V)\subset U$, for all $n\in\mathbb{N}$.
We say that $\Lambda$ is {\em neutral} if $\Lambda=\Lambda^+\cap \Lambda^-$
where $\Lambda^\pm$ is a Lyapunov stable set for $f^{\pm1}$.

The following can be proved as in Lemma 2.2 of \cite{cmp}.

\begin{lemma}
 \label{neutral}
Let $\Lambda$ be a neutral subset of a surface diffeomorphism $f$. If $x\in M$ satisfies
$\omega(x)\cap \Lambda\neq\emptyset$, then $\omega(x)\subset \Lambda$.
If $x_k$ is a sequence of periodic point converging to some point in $\Lambda$, then
$\Lambda$ contains any Hausdorff limit of the sequences formed by the orbits of $x_k$ under $f$. 

\end{lemma}

A {\em cycle} of a surface diffeomorphism $f$ is a finite collection of homoclinic classes
$H_1,\cdots,H_r$ of $f$ with $H_1=H_r$ such that
$$
(W^s_f(H_i)\setminus H_i)\cap (W^s_{f^{-1}}(H_{i+1})\setminus H_{i+1})\neq\emptyset,
\quad\quad\forall 1\leq i\leq r-1.
$$
The next result was proved in Section 3 of \cite{cmp}.

\begin{lemma}
 \label{homoneutral}
There is a residual subset of surface diffeomorphisms $\mathcal{R}_2$
such that if $f\in \mathcal{R}_2$, then $f$ has no cycles and
every homoclinic class of $f$ is neutral.
\end{lemma}

About the residual subset $\mathcal{R}_2$ in Lemma \ref{homoneutral} we
obtain following lemma. 

\begin{lemma}
\label{cm}
If $f\in\mathcal{R}_2$, then every hyperbolic homoclinic class $H$ of $f$ satisfies
$
H\cap \cl(\sink(f))=\emptyset.
$
\end{lemma}

\begin{proof}
Take a hyperbolic homoclinic class $H$ of some $f\in\mathcal{R}_2$ and suppose
by contradiction that $H\cap \cl(\sink(f))\neq\emptyset$.
Then,
there is a sequence $x_k$ of sinks converging to some $p\in H$.
Since $f\in\mathcal{R}_2$ we have that
$H$ is neutral, and so, by Lemma \ref{neutral}, it contains any Hausdorff limit of the sequence formed by the orbits of $x_k$.
Since $H$ is hyperbolic by hypothesis, and a hyperbolic set cannot contain infinitely many orbits of
sinks in a neighborhood of it, we obtain a contradiction proving the result.
\end{proof}

\subsection{Quasi-attracting sets and dissipative attractors}
A {\em quasi-attracting set} of a surface diffeomorphism $f$ is a compact invariant set
$\Lambda$ for which there is a nested sequence of of open neighborhoods $U_n$
such that $f(\cl(U_n))\subset U_n$ and $\Lambda=\bigcap_{n\in\mathbb{N}}U_n$.
Every attractor is a quasi-attracting set which in turns is Lyapunov stable.

We denote by $\sad(f)$ the set of saddles of $f$.
We say that $x\in\sad(f)$ is {\em dissipative} if
$|\det(Df^{n_{x,f}}(x))|<1$.
Denote by $\sad_d(f)$ the set of
dissipative saddles.

An attractor $A$ is {\em dissipative} if 
$A=H_f(x)$ for some $x\in\sad_d(f)$.

We will study quasi-attracting set first and, then, apply the result to obtain dissipativity for non-trivial attractors.
For this we shall use the classical result below \cite{F}.

\begin{lemma}[Franks's Lemma]
\label{frank}
For every $f\in \mbox{Diff}^1(M)$ where $M$ is a closed manifold
and every neighborhood $W(f)\subset \mbox{Diff}^1(M)$
of $f$ there exist $\epsilon>0$ and a neighborhood $W_0(f)\subset
W(f)$ of $f$ such that if $h\in W_0(f)$, $\{x_0,\cdots ,x_{n-1}\}\subset M$
is a finite set ($n\in\mathbb{N}^+$), $U\subset M$ is a neighborhood of $\{x_0\cdots ,x_{n-1}\}$
and $L_i:T_{x_i}M\to T_{h(x_i)}M$ are linear maps satisfying
$\| L_i-Dh(x_i)\| <\epsilon$ ($\forall i=0,\cdots ,n-1$),
then there is $g\in W(f)$ such that $g(x)=h(x)$
in $\{x_0,\cdots ,x_{n-1}\}\cup (M\setminus U)$, and
$Dg(x_i)=L_i$ for every $i=0,\cdots,n-1$.
\end{lemma}

Franks's Lemma is used in the proof below.

\begin{lemma}
\label{move-attractor}
There is a residual subset $\mathcal{R}_6$ of surface diffeomorphisms
$f$ for which every quasi-attracting set $A$ satisfies
$$
A\cap(\cl(\sad_d(f))\cup\cl(\sink(f)))\neq\emptyset.
$$
\end{lemma}

\begin{proof}
Define $S:\diff^1(M)\to 2^M_c$ by
$$
S(f)=\cl(\sad_d(f))\cup\cl(\sink(f))
$$
Clearly $S$ is lower-semicontinuous, and so,
it is upper-semicontinuous in a residual subset $\mathcal{A}$ of $\diff^1(M)$.

By a well-known consequence of the Ergodic Closing Lemma \cite{M} (c.f. Theorem 3.8 p.13 in \cite{abc})
there is another residual subset $\mathcal{B}$ of surface diffeomorphisms $f$
such that  for every  ergodic measure $\mu$ of
$f$
there are sequences $g_k\to f$ and $p_k\in\sad(g_k)$ such that the Dirac measure supported on
the orbit $\gamma_k$ of $p_k$ under $g_k$ converges to $\mu$ with respect to the weak-* topology.

Define $\mathcal{R}_6=\mathcal{A}\cap\mathcal{B}$.
Then, $\mathcal{R}_6$ is a residual subset of surface diffeomorphisms.

Now, take $f\in\mathcal{R}_6$
and assume by contradiction that
$$
A\cap\cl(\sad_d(f))\cup\cl(\sink(f)))=\emptyset
$$
for some quasi-attracting set $A$ of $f$.
Since $S$ is upper-semicontinuous on $\mathcal{A}$, we can arrange
neighborhoods $U$ of $A$ and $W(f)$ of $f$ such that
\begin{equation}
\label{contra}
U\cap(\sad_d(\tilde{g})\cup\sink(\tilde{g}))=\emptyset,
\quad\quad\forall \tilde{g}\in W(f).
\end{equation}

Put this $W(f)$ in the Franks's Lemma to obtain the neighborhood
$W_0(f)\subset W(f)$ of $f$ and $\epsilon>0$.
Set
$$
C=\sup\{\|Dg\|:g\in W(f)\}
$$
and fix $\delta>0$ such that
$$
|1-e^{-\frac{\delta}{2}}|<\frac{\epsilon}{C}.
$$

It is known (Lemma 3.1 in \cite{po}) that
$f$ has an ergodic invariant measure
$\mu$ supported on $A$ such that
$$
\int\log(|\det f|)d\mu\leq0.
$$

Since $f\in\mathcal{B}$, we can take sequences $g_k\to f$ and $p_k\in\sad(g_k)$ such that the Dirac measure supported on
the orbit $\gamma_k$ of $p_k$ under $g_k$ converges to $\mu$.

Since $\int\log(|\det f|)d\mu\leq0$, we can fix $k$ such that
$$
p_k\in U, \quad g_k\in W_0(f)\quad\mbox{ and }\quad|\det Dg_k^{n_{p_k,g_k}}(p_k)|<e^{n_{p_k,g_k}\delta}.
$$
Once we fix this $k$ we define the linear maps
$L_i: T_{g_k^i(p_k)}M\to T_{g_k^i(p_k)}M$, $0\leq i\leq n_{p_k,g_k}-1$ by
$L_i=e^{-\frac{\delta}{2}}Dg_k(g^i_k(p_k))$.
It follows that
$$
\|L_i-Dg_k(g^i_k(p_k))\|\leq |1-e^{-\frac{\delta}{2}}|\cdot C<\epsilon,
\quad\quad\forall 0\leq i\leq n_{p_k,g_k}-1.
$$
Since $g\in W_0(f)$, we can apply the Franks's Lemma to
$x_i=g_k^i(p_nk)$, $0\leq i\leq n_{p_k,g_k}-1$, in order to obtain a
diffeomorphism $\tilde{g}\in W(f)$
such that
$\tilde{g}=g_k$ along the orbit of $p_k$ under $g_k$ (thus $p_k$ is a periodic point of $\tilde{g}$ with $n_{p_k,\tilde{g}}=n_{p_k,g_k}$)
and $D\tilde{g}(\tilde{g}^i(p_k))=L_i$ for $0\leq i\leq n_{p_k,g_k}-1$.
Consequently,
$$
D\tilde{g}^{n_{p_k,\tilde{g}}}(p_k)=\displaystyle\prod_{i=0}^{n_{p_k,g_k}-1}L_i.
$$
A direct computation then shows that
$$
|\det(D\tilde{g}^{n_{p_k,\tilde{g}}}(p_k))|
=e^{-n_{p_k,g_k}\delta} |\det Dg_k^{n_{p_k,g_k}}(p_k)|<1.
$$
Up to a small perturbation if necessary we can assume
that $p_k$ has no eigenvalues of modulus $1$.
Then, we have $p_k\in \sad_d(\tilde{g})\cup\sink(\tilde{g})$
by the previous inequality yielding
$p_k\in U\cap(\sad_d(\tilde{g})\cup\sink(\tilde{g}))$.
Since $\tilde{g}\in W(f)$, we obtain a contradiction by (\ref{contra}) and
the result follows.
\end{proof}

An application of the above lemma is given below.

\begin{corollary}
\label{prove-attractor}
There is a residual subset $\mathcal{R}_7$ of surface diffeomorphisms $f$
for which every non-trivial attractor is dissipative.
In particular,  they are all contained in $\cl(\sad_d(f))\setminus \cl(\sink(f))$.
\end{corollary}

\begin{proof}
Let  $\mathcal{R}_2$ and $\mathcal{R}_6$ be the residual subsets in lemmas \ref{neutral} and \ref{move-attractor}
respectively.
By Lemma \ref{homoneutral} we have that
every non-trivial attractor of $f$ is a homoclinic class and, furthermore,
the homoclinic classes of $f$ are pairwise disjoint.

Define $\mathcal{R}_7=\mathcal{R}_2\cap\mathcal{R}_6$.
Then, $\mathcal{R}_7$ is a residual subset of surface diffeomorphisms.

Now, take $f\in\mathcal{R}_7$ and a non-trivial attractor $A$ of
$f$.
Since $f\in\mathcal{R}_6$, and every attractor is a quasi-attracting set,
we have $A\cap(\cl(\sad_d(f))\cup\cl(\sink(f)))\neq\emptyset$ by Lemma \ref{move-attractor}.
Since $A$ is non-trivial, we actually have
$A\cap\cl(\sad_d(f))\neq\emptyset$.
Since $A$ is an attractor, we conclude that there is $x\in A\cap \sad_d(f)$,
and so, $H_f(x)\subset A$.
Since $f\in \mathcal{R}_2$, Lemma \ref{homoneutral} implies that $A$ is also a homoclinic class
and, then, $A=H_f(x)$.

Since $x\in\sad_d(f)$, we obtain $H_f(x)\subset\cl(\sad_d(f))$ by the Birkhoff-Smale theorem \cite{hk}.
Then, $A\subset \cl(\sad_d(f))\setminus\cl(\sink(f))$
and the proof follows.
\end{proof}

\subsection{Dissipative presaddles}
We say that $x$ is a {\em dissipative presaddle} of a surface diffeomorphism $f$
if there are sequences $f_k\to f$ and $x_k\in \sad_d(f_k)$ such that
$x_k\to x$. This definition is similar to that in \cite{w0}.
Denote by $\sad^*_d(f)$ the set of dissipative presaddles of $f$.

The following lemma says that the set of dissipative presaddles
does not explode. Its proof is a direct consequence of the definition.

\begin{lemma}
 \label{l5}
For every surface diffeomorphism $g$ and every neighborhood $U$ of $\sad_d^*(g)$ there is a neighborhood $\mathcal{V}_g$ of $g$ such that
$\sad_d^*(h)\subset U$, for every $h\in \mathcal{V}_g$.
\end{lemma}

Another interesting property is that the set of dissipative presaddles contains the set of accumulation points of the sinks.
More precisely, we have the following result.

\begin{lemma}
\label{l9}
Every Kupka-Smale surface diffeomorphism $f$ satisfies
$$
\cl(\sink(f))\setminus\sink(f)\subset \sad_d^*(f).
$$
\end{lemma}

\begin{proof}
Take $p\in \cl(\sink(f))\setminus \sink(f)$. Take arbitrary neighborhoods $W(f)$ and $U$ of $f$ and $p$ respectively. 
Without loss of generality we can assume $W(f)\subset K(f)$.
Put this $W(f)$ in the Franks lemma to obtain the neighborhood $W_0(f)\subset W(f)$ of $f$ and $\epsilon>0$.

As $p\in Cl(\sink(f))\setminus \sink(f)$,
there is a sequence of sinks converging to $p$.

Since $f$ is Kupka-Smale, the associated sequence of periods must be unbounded.

Then, we can apply
a result of Pliss \cite{Pl} in order to choose $x\in \sink(f)\cap U$ and
a sequence of linear isomorphisms $L_i: T_{f^i(x)}M\to T_{f^{i+1}(x)}M$
with $\|L_i-Df(f^i(x))\|\leq \epsilon$, for every $0\leq i\leq n_x-1$, such that
$\displaystyle\prod_{i=0}^{n_x-1}L_i$ has an eigenvalue of modulus $1$.

For all $0\leq t\leq 1$ and $0\leq i\leq n_x-1$ we define $L_i^t: T_{f^i(x)}M\to T_{f^{i+1}(x)}M$ by
$$
L_i^t=(1-t)Df(f^i(x))+tL_i.
$$
Since
$$
\|L_i^t-Df(f^i(x))\|=t\|L_i-Df(f^i(x))\|\leq \epsilon,
\quad\quad\forall 0\leq t\leq 1,0\leq i\leq n_x-1,
$$
and $h=f\in W_0(f)$, we can apply the Franks lemma to
$x_i=f^i(x)$, $0\leq i\leq n_x-1$, in order to obtain
a one parameter family of surface diffeomorphisms
$g_t\in W(f)$, with $0\leq t\leq 1$ and $0\leq i\leq 1$, such that $g_t=f$ along the orbit of $x$ under $f$
(thus $x$ is a periodic point of $g_t$ with $n_{x,g_t}=n_x$)
and $Dg_t(g^i_t(x))=L_i^t$ for all $0\leq t\leq 1$ and $0\leq i\leq n_x-1$.
Consequently,
$$
Dg^{n_{x,g_t}}_t(x)=\displaystyle\prod_{i=0}^{n_x-1}L_i^t,
\quad\quad\forall 0\leq t\leq 1.
$$
As $L_i^t$ depends continuously on $t\in [0,1]$, for every $0\leq i\leq n_x-1$, the above identity implies that the curve of linear operators
$$
\mathcal{C}(t)=Dg^{n_{x,g_t}}_t(x), \quad\quad0\leq t\leq 1,
$$
is continuous.
As $Dg_0^{n_{x,g_0}}(x)=Df^{n_x}(x)$ and $Dg_1^{n_{x,g_1}}(x)=\displaystyle\prod_{i=1}^{n_x-1}L_i$,
we have $x\in\sink(g_0)$ and $x\not\in \sink(g_1)$.
It follows that
$$
O=\{t\in [0,1]:x\in \sink(g_t)\}.
$$
is a nonempty proper subset of $[0,1]$.
Furthermore, the continuity of $\mathcal{C}$ implies that $O$ is open in $[0,1]$, and so,
a countable union of open intervals.
Take any boundary point $t_0$ of one of these intervals.

The continuity of $\mathcal{C}$ implies that the eigenvalues of $Dg_{t_0}^{n_{x,g_{t_0}}}(x)$ have
modulus $\leq1$.

If both eigenvalues have modulus less than $1$, the same continuity implies that
$t_0$ is an interior point of $O$ which is absurd.
Then, at least one of these eigenvalues have modulus $1$.

If both eigenvalues are complex of modulus $1$, then we can apply the results in p. 1243 of \cite{mamo}
in order to find a surface diffeomorphism $h_1$ close to $g_{t_0}$ (and thus within $W(f)$)
having a dissipative saddle close to $x$ (and thus within $U$). From this we get $\sad_d(h_1)\cap U\neq\emptyset$.

If both eigenvalues are real with modulus $1$, or if 
the two eigenvalues have modulus less than and equal to $1$ respectively,
then we can easily find a surface diffeomorphism $h_2$ close to $g_{t_0}$ (and thus within $W(f)$) satisfying $x\in \sad_d(h_3)$.

As $W(f)$ and $U$ were chosen arbitrary, we conclude that $p\in \sad^*_d(f)$ proving the result.
\end{proof}

\subsection{Finiteness of dissipative attractors and dominated splittings}
The following lemma about dominated splittings will be useful to prove finiteness of dissipative non-trivial hyperbolic attractors.
Its proof uses an argument given in the Araujo's thesis \cite{A} (which was outlined in the proof of Corollary 3.3 p.981 in \cite{PS}).

\begin{lemma}
\label{stable-manifold}
Let $\Lambda$ be a compact invariant set with a dominated splitting $T_\Lambda M=E_\Lambda\oplus F_\Lambda$
of a surface diffeomorphism $f$.
Then, for every $p\in \Lambda$ exhibiting $0<\gamma<1$ such that
\begin{equation}
\label{ali}
\displaystyle\prod_{i=0}^{l-1}\|Df(f^i(p))/E_{f^i(p)}\|\leq\gamma^l,
\quad\quad\forall l\in \mathbb{N}^+
\end{equation}
there are $\epsilon>0$ and a $C^1$ embedding $\phi: [-\epsilon,\epsilon]\to M$ with $\phi(0)=p$
such that the $C^1$ submanifold $W_\epsilon(p)=\phi([-\epsilon,\epsilon])$ satisfies the following properties:
\begin{enumerate}
\item
$T_p W_\epsilon(p)=E_p$;
\item
there is $N\in\mathbb{N}^+$ such that
$$
\lim_{n\to\infty}d(f^{nN}(x),f^{nN}(p))=0,
\quad\quad\forall x\in W_\epsilon(p).
$$
\end{enumerate}
\end{lemma}

\begin{proof}
Denote by $\emb^1_\Lambda(I,M)$ the space of $C^1$ embeddings $\beta:I=[-1,1]\to M$ satisfying $\beta(0)\in\Lambda$.
We equipp $\emb^1_\Lambda(I,M)$ with the $C^1$ topology.
It follows easily that $\emb^1_\Lambda(I,M)$ is a firber bundle over $\Lambda$
with projection $\pi(\beta)=\beta(0)$.

By well known properties of dominated splittings  (\cite{hps} or the proof of Proposition 2.3 p.386 in \cite{MM})
there is a section $\phi:\Lambda\to \emb^1_\Lambda(I,M)$ such that
defining
$W_\epsilon(x)=\phi(x)([-\epsilon,\epsilon])$, $\forall 0<\epsilon\leq 1$, we have the following properties for all $x\in \Lambda$:

\begin{itemize}
\item
$T_xW_1(x)=E_x$;
\item
There is a power $g=f^N$ of $f$ such that for every $0<\epsilon_1<1$ there is $0<\epsilon_2<1$ satisfying
$
g(W_{\epsilon_2}(x))\subset W_{\epsilon_1}(g(x)).
$
\end{itemize}

The chain rule and (\ref{ali}) imply
\begin{equation}
\label{nova}
\displaystyle\prod_{i=0}^{l-1}\|Dg(g^i(p))/E_{g^i(p)}\|\leq(\gamma^{N})^l,
\quad\quad\forall l\in \mathbb{N}^+.
\end{equation}

Fix $0<\epsilon_2<1$ such that
$$
g(W_{\epsilon_2}(z))\subset W_{\frac{1}{2}}(g(z)),\quad\quad\forall z\in \Lambda.
$$
Since $\phi$ is a section (in particular continuous), we get easily
that there is $\delta>0$ such that
$W_1(z)\cap B_\delta(z)\subset W_{\epsilon_2}(z)$ for all $x\in \Lambda$.

\vspace{5pt}

Let us prove that every $x\in\Lambda$, $y\in W_{\frac{1}{2}}(x)$  and $n\in\mathbb{N}$
satisfying
$d(g^j(x),g^j(y))\leq\delta$, $\forall 0\leq j\leq n$, also satisfy
\begin{equation}
\label{aline}
g^j(y)\in W_{\frac{1}{2}}(g^j(x)), \quad\quad\forall 0\leq j\leq n.
\end{equation}

Indeed, we have that $y\in W_1(x)\cap B_\delta(x)$ so $y\in W_{\epsilon_2}(x)$ thus $g(y)\in W_{\frac{1}{2}}(g(x))$.
Suppose $g^j(y)\in W_{\frac{1}{2}}(g^j(x))$ for some $0\leq j\leq n-1$.
Then, $g^j(y)\in W_1(g^j(x))\cap B_\delta(g^j(x))$ so $g^j(y)\in W_{\epsilon_2}(g^j(x))$ thus
$g^{j+1}(y)\in W_{\frac{1}{2}}(g^{j+1}(x))$. Therefore, we get (\ref{aline}) by induction.

\vspace{5pt}

Let us continue with the proof.

Denote $\tilde{E}(z)=T_zW_{\frac{1}{2}}(x)$, $\forall x\in\Lambda$ and $z\in W_{\frac{1}{2}}(x)$.
Shrinking $\delta$ if necessary we can assume
\begin{equation}
\label{ganzo}
\| Dg(z_1)/\tilde{E}_{z_1}\|\leq (1+c)\|Dg( z_2)/\tilde{E}_{z_2}\|
\end{equation}
for every $z_1,z_2$ with $d(z_1,z_2)<\delta$.

Take $0<\epsilon<\frac{1}{2}$ small enough so that
$\it{l}(W_\epsilon(x))<\delta$, for all $x\in \Lambda$, where $\it{l}(\cdot)$ denotes the length operation.

Fix $c>0$ such that $\gamma_1=(1+c)\gamma^N<1$.

We claim that $p$ as in (\ref{nova}) satisfies
\begin{equation}
\label{gato}
(a)\quad\it{l}(g^j(W_\epsilon(p)))\leq \gamma_1^j\it{l}(W_\epsilon(p))
\mbox{ and } (b)\quad
g^j(W_\epsilon(p))\subset W_{\frac{1}{2}}(g^j(p)),
\quad\forall j\in \mathbb{N}.
\end{equation}
Indeed,
the assertion is trivial for $j=0$. Now,
assume by induction that there is $n\in\mathbb{N}^+$ such that the assertion is true for
every $0\leq j\leq n$.

Note that
\begin{equation}
\label{aldro}
\it{l}(g^{n+1}(W_\epsilon(p)))\leq
\sup\{\|Dg^{n+1}(z)/\tilde{E}_z\|:z\in W_\epsilon(p)\}\it{l}(W_\epsilon(p))
=
\end{equation}
$$
\sup\left\{
\prod_{j=0}^{n}\|Dg(g^j(z))/\tilde{E}_{g^j(z)}\|:z\in W_\epsilon(p)\right\}\it{l}(W_\epsilon(p)).
$$
Moreover, if $z\in W_\epsilon(p)$, then $g^j(z)$ belongs to the curve
$g^j(W_\epsilon(p))$ which, in turns, contains $g^j(p)$ and has length $\leq \gamma_1^j\delta<\delta$
for all $0\leq j\leq n$ by the induction hypothesis.
Therefore, $d(g^j(z),g^j(p))<\delta$ for $0\leq j\leq n$ thus (\ref{ganzo}) yields
$$
\|Dg(g^j(z))/\tilde{E}_{g^j(z)}\|\leq(1+c)\|Dg(g^j(p))/\tilde{E}_{g^j(p)}\|, \quad\quad\forall 0\leq j\leq n.
$$
Replacing in (\ref{aldro}) and using (\ref{nova}) we get
$$
\it{l}(g^{n+1}(W_\epsilon(p)))\leq
(1+c)^{n+1}\left(\prod_{j=0}^n\|Dg(g^j(p))/E_{g^j(p)}\|\right)\it{l}(W_\epsilon(p))\leq
$$
$$
((1+c)\gamma^N)^{n+1}\it{l}(W_\epsilon(p))=\gamma_1^{n+1}\it{l}(W_\epsilon(p)).
$$

\vspace{5pt}

This proves (\ref{gato})-(a) for $n+1$ and, in particular, $\it{l}(g^{n+1}(W_\epsilon(p)))<\delta$.

Since $g^{n+1}(W_\epsilon(p))$ is a curve containing
both $g^{n+1}(z)$ and $g^{n+1}(p)$ we conclude that
$d(g^{n+1}(z),g^{n+1}(p))\leq \delta$, $\forall z\in W_\epsilon(p)$.

But we also have
$d(g^j(z),g^j(p))\leq\delta$ ($\forall z\in W_\epsilon(p)$ and $0\leq j\leq n$) by induction. So,
$d(g^j(z),g^j(p))\leq\delta$, $\forall z\in W_\epsilon(p)$ and $0\leq j\leq n+1$.

Since $0<\epsilon<\frac{1}{2}$, we also have $z\in W_{\frac{1}{2}}(p)$ for all $z\in W_\epsilon(p)$. Therefore
$g^{n+1}(z)\in W_{\frac{1}{2}}(g^{n+1}(p))$, $\forall z\in W_\epsilon(p)$, by (\ref{aline}).
This proves (\ref{gato})-(b) for $n+1$, and so, (\ref{gato}) holds by induction.

Now we observe that (\ref{gato}) implies
the second conclusion of the lemma whereas the first is trivial.
This ends the proof. 
\end{proof}

We will need the Pliss's Lemma (c.f. Lemma 3.0.1 in \cite{PS}).

\begin{lemma}[Pliss's Lemma]
\label{pliss}
For every $f\in \diff^1(M)$ and $0<\gamma_1<\gamma_2$ there are
$m\in \mathbb{N}^+$ and $c>0$ such that for any $x\in M$, any subspace $S\subset T_xM$ and any integer $n\geq m$
satisfying
$$
\prod_{i=1}^n\|Df(f^i(x))/S_i\|\leq \gamma_1^n,
$$
with $S_i=Df^i(x)(S)$,
there are $0\leq n_1<n_2<\cdots<n_l\leq n$ with $l\geq cn$ such that
$$
\prod_{i=n_r}^j\|Df(f^i(x))/S_i\|\leq \gamma_2^{j-n_r},
\quad \forall r=1,\cdots,l,\forall j= n_r,\cdots,n.
$$
\end{lemma}

Combining the Pliss's Lemma with Lemma \ref{stable-manifold} we obtain the following result.

\begin{theorem}
\label{finiteness-attractors}
The number of dissipative non-trivial hyperbolic attractors contained in a compact invariant set with
a dominated splitting of a Kupka-Smale surface diffeomorphism is finite.
\end{theorem}

\begin{proof}
Assume by contradiction that
there is compact invariant set $\Lambda$
with a dominated splitting
$T_\Lambda M=E_\Lambda\oplus F_\Lambda$, of a Kupka-Smale diffeomorphism $f$,
containing an infinite sequence of dissipative non-trivial hyperbolic attracttors $A_k$
of $f$

Let $K$ and $0<\lambda<1$ the constants associated to the splitting,
and $x_k\in \sad_d(f)$ be a sequence satisfying $A_k=H_f(x_k)$, $\forall k\in\mathbb{N}$.

It follows from the unicity of dominated splittings that
$E_{x_k}\oplus F_{x_k}=E^s_{x_k}\oplus E^u_{x_k}$ for every $k$.
From this we obtain
$$
\frac{|\lambda(x_k,f)|}{|\sigma(x_k,f)|}=\frac{\|Df^{n_{x_k}}(x_k)/E_{x_k}\|}{\|Df^{n_{x_k}}(x_k)/F_{x_k}\|}
\leq K\lambda^{n_{x_k}},
\quad\quad\forall n\in \mathbb{N}.
$$
But $|\lambda(x_k,f)\sigma(x_k,f)|<1$ since $x_k\in\sad_d(f)$, so
$$
|\lambda(x_k,f)|^2=\frac{|\lambda(x_k,f)|}{|\sigma(x_k,f)|}|\lambda(x_k,f)\sigma(x_k,f)|
\leq K\lambda^{n_{x_k}},
$$
yielding
$$
|\lambda(x_k,f)|\leq K_0\gamma_0^{n_{x_k}},
\quad\quad\forall k\in\mathbb{N},
$$
where $K_0=\sqrt{K}$ and $\gamma_0=\sqrt{\lambda}$ (thus $K_0>0$ and $0<\gamma_0<1$).

Since $f$ is Kupka-Smale, we obtain that the period sequence $n_{x_k}$ is unbounded.

Therefore, up to passing to a subsequence if necessary,
we can assume $n_{x_k}\to\infty$.

Choose $\gamma_0<\gamma_1<1$. Since $n_{x_k}\to\infty$ we can choose $k_0\in\mathbb{N}$ so that
$(\gamma_0^{-1}\gamma_1)^{n_k}\geq K_0$ for every $k\geq k_0$.
Then, $K_0\gamma_0^{n_{x_k}}\leq \gamma_1^{n_k}$ yielding
\begin{equation}
\label{lambda}
|\lambda(x_k,f)|\leq \gamma_1^{n_{x_k}},
\quad\quad\forall k\geq k_0.
\end{equation}

Fix $0<\gamma_1<\gamma_2<1$.
Choose $N\in\mathbb{N}^+$ and $c>0$ as in the Pliss's Lemma.
Clearly there is an integer $k_1>k_0$ such that $n_{x_k}\geq N$ for every $k\geq k_1$.
Denoting $S_i^k=Df^i(x_k)(E^s_{x_k})=E^s_{f^i(x_k)}$ we obtain from (\ref{lambda}) and the chain rule that
$$
\displaystyle\prod_{i=0}^{n_{x_k}-1}\|Df(f^i(x_k))/S^k_i\|\leq \gamma_1^{n_k},
\quad\quad\forall k\geq k_1.
$$
Then, by the Pliss's Lemma, for all $k\geq k_1$ we can select
positive integers $l_k\geq c\cdot n_{x_k}$ and
$0\leq n_1^k<n_2^k<\cdots<n_{l_k}^k\leq n_{x_k}$ such that
\begin{equation}
\label{infinity}
\displaystyle\prod_{i=n_r^k}^j
\|Df(f^i(x_k))/S^k_i\|\leq \gamma_2^{j-n^k_r},
\quad\forall k\geq k_1,\forall r=1,\cdots, l_k, \forall j=n_r^k,\cdots,n_{x_k}-1.
\end{equation}
By compactness we can assume $f^{n^k_1}(x_k)\to p$ (for some $p$) and so $E_{f^{n^k_1}(x_k)}=E^s_{f^{n^k_1}(x_k)}\to E_{p}$. 

We claim that
\begin{equation}
\label{infinityy}
\displaystyle\prod_{i=0}^{l-1}\| Df(f^i(p))/E_{f^i(p)}\|\leq K_2\gamma_2^l,
\quad\quad\forall l\in\mathbb{N},
\end{equation}
where $K_2=\gamma_2^{-1}$.

Indeed, fix $l\in\mathbb{N}^+$.
We clearly have $n_{x_k}-n_1^k\geq l_k$ so
$n_{x_k}-1-n_1^k\geq l_k-1\geq cn_{x_k}-1$.
As
$c>0$ we have $cn_{x_k}-1\to\infty$ whence $n_{x_k}-1-n_1^k\to\infty$ too.
Therefore, there is $k_l\in\mathbb{N}$ such that
for every $k\geq k_l$ one has
$l=j-n_1^k$ for some
$n_1^k\leq j\leq n_{x_k}-1$.

So, (\ref{infinity}) implies
$$
\gamma_2^l=\gamma_2^{j-n_1^k}\geq \displaystyle\prod_{i=n_1^k}^j\|Df(f^i(x_k))/S_i^k\|
=
\displaystyle\prod_{i=n_1^k}^{l+n_1^k}\|Df(f^i(x_k))/E^s_{f^i(x_k)}\|
=
$$
$$
\displaystyle\prod_{i=0}^l\|Df(f^i(f^{n_1^k}(x_k)))/E^s_{f^i(f^{n_1^k}(x_k))}\|,
$$
i.e.,
$$
\displaystyle\prod_{i=0}^l\|Df(f^i(f^{n_1^k}(x_k)))/E^s_{f^i(f^{n_1^k}(x_k))}\|\leq \gamma_2^l,
\quad \quad\forall k\geq k_l.
$$
Since $l$ is fixed we can take limits as $k\to\infty$ above to obtain (\ref{infinityy}).

Replacing $\gamma_2$ by some $\gamma_2<\gamma<1$
we can assume that
$$
\displaystyle\prod_{i=0}^{l-1}\| Df(f^i(p))/E_{f^i(p)}\|\leq \gamma^l,
\quad\quad\forall l\in\mathbb{N}.
$$
This estimative together with Lemma \ref{stable-manifold}
implies that there are $\epsilon>0$ and embeddding $\phi:[-\epsilon,\epsilon]\to M$
with $\phi(0)=p$
such that the $C^1$ submanifold $W_\epsilon(p)=\phi([-\epsilon,\epsilon])$ satisfies
\begin{enumerate}
 \item $T_p W_\epsilon(p)=E_p$;
 \item there is $N\in\mathbb{N}^+$ such that $\lim_{n\to\infty}d(f^{nN}(x),f^{nN}(p))=0$
for every $x\in W_\epsilon(p)$.
\end{enumerate}

We assert that
$A_k\cap W_\epsilon(p)$ for every $k$ large.

The proof uses an inevitable intersection argument described as follows.

\vspace{5pt}

Take a small product neighborhood $V=[-1,1]\times [-1,1]$ around $p$, with $p=(0,0)$,
such that the curves $0\times [-1,1]$ and $[-1,1]\times 0$ are parallel to $F_p$ and $E_p$
respectively. Since $T_pW_\epsilon(p)=E_p$, we can assume without loss of generality that $W_\epsilon(p)=[-1,1]\times 0$ and that
$0\times [-1,1]$ is parallel to $F$.

Let $W^u(x,k)$ denote the unstable manifold through $x\in A_k$, $k\in\mathbb{N}$.
Since each $A_k$ is a non-trivial hyperbolic attractor, we have that the unstable manifold $W^u(x)$
through $x\in A_k$ is a smooth immersed curve everywhere tangent to $F$.

To simplify we write
$p_k=f^{n^k_1}(x_k)$ thus $p_k\in A_k$ and $p_k\to p$. Denote by $C_k$ 
the connected component of $W^u(p_k)\cap V$ containing $p_k$.
It follows that $C_k$ is a sequence of smooth curves
everywhere tangent to $F$
approaching to $p$.

Taking $k$ large we have that $C_k$ contains the point $p_k$ close to $p$.
Since $0\times [-1,1]$ is parallel to $F$, the slope $\alpha_k$ between $C_k$ and $0\times [-1,1]$
goes to $0$ as $k\to\infty$.
Take any boundary point $b_k$ of $C_k$. We have that $b_k\in A_k$, since $C_k\subset A_k$, whence $W^u(b_k)$ is well-defined and contained in $A_k$.
If $b_k\in \inte(V)$ we can use $W^u(b_k)$ to extent $C_k$ around $b_k$. But this is impossible because $b_k$ is a boundary point,
so $b_k\in \partial V$. Since $\alpha_k\to 0$, $b_k\in [-1,1]\times \{-1,1\}$ for $k$ large.
Since $C_k$ is everywhere tangent to $F$ (and so transversal to $E$)
we obtain that the two boundary points of $C_k$ belong to different connected components of $[-1,1]\times \{-1,1\}$.
Since $C_k$ is a continuous curve, we conclude that $C_k\cap ([-1,1]\times 0)\neq\emptyset$.
As $W_\epsilon(p)=[-1,1]\times 0$, we obtain $C_k\cap W_\epsilon(p)\neq\emptyset$
yielding $A_k\cap W_\epsilon(p)$ for $k$ large. This completes the proof.

\vspace{5pt}

Using the assertion we can take $k_1\neq k_2$, and points
$z_i\in A_{k_i}\cap W_\epsilon(p)$ ($i=1,2$).

Since $z_i\in W_\epsilon(p)$ we obtain that
$\omega_{f^N}(z_1)=\omega_{f^N}(z_2)=\omega_{f^N}(p)$.
But $z_i\in A_{k_i}$ so $\omega_{f^N}(z_i)\subset A_{k_i}$ for $i=1,2$.
Therefore, $\omega_{f^N}(p)\subset A_{k_1}\cap A_{k_2}$ which is impossible since
$\omega_{f^N}(p)\neq\emptyset$ and $A_{k_1}\cap A_{k_2}=\emptyset$.
This contradiction concludes the proof.
\end{proof}

\section{degenerate points}
\label{sec2}

\noindent
We shall define degenerate points for surface diffeomorphisms.
The basic properties of these points are that:

\begin{itemize}
\item
they cannot have any dominated splitting (Proposition \ref{arrasa}), but there is domination outside them (see Proposition \ref{ladilla2}).
\item
they can be turned  into dissipative homoclinic tangencies by small perturbations (Proposition \ref{ladi1});
\end{itemize}

The first property will be very easy to prove. The proof of the last two properties
involves the use of the techniques in \cite{PS} and \cite{w}.
For this reason we will give them in a separated section (Section \ref{sec-final}).

\subsection{Definition and abscence of domination}
We will need the following terminology.
The orthogonal complement of a linear subspace $E$ of $\mathbb{R}^2$ is denoted by
$E^\perp$.
The {\em angle} between linear spaces $E,F$ of $\mathbb{R}^2$ is defined by $\angle(E,F)=\|L\|$, where $L: E\to E^\perp$ is the linear operator
satisfying $F=\{u+L(u):u\in E\}$.

Given a surface diffeomorphism $f$
and $x\in\sad(f)$,
we denote by
$E^s_x$ and $E^u_x$ the eigenspaces corresponding to the eigenvalues
of modulus less and bigger than $1$ of $Df^{n_x}(x)$
respectively. Notation $E^{s,f}_x$ and $E^{u,f}_x$ will indicate dependence on $f$.
Define,
$$
\angle(x,f)=\min\{\angle(E^s_{f^i(x)},E^u_{f^i(x)}):0\leq i\leq n_x-1\},
\quad\quad\forall x\in\sad(f).
$$

\begin{defi}
\label{critical}
We say that $x$ is a {\em degenerate point} of $f$ if there are
sequences $f_k\to f$ and $x_k\in \sad_d(f_k)$ such that
$x_k\to x$ and $\angle(x_k,f_k)\to0$.
Denote by $\de(f)$ the set of degenerate points of $f$.
\end{defi}

At first glance we can observe that $\de(f)$ is a possibly empty compact invariant set contained in $\sad_d^*(f)$.
A very basic property of these points is that they do not have any domination.
Indeed, we have the following result.

\begin{proposition}
 \label{arrasa}
If $f$ is a surface diffeomorphism, then $\de(f)$ cannot have a dominated splitting.
\end{proposition}

\begin{proof}
Suppose by contradiction that $\de(f)$ has a dominated splitting.
It follows from well-known properties of dominated splittings that
there are $\alpha>0$, and neighborhoods $U$ of $\de(f)$ and $\mathcal{U}$ of $f$
such that $\angle(E^{s,g}_x,E^{u,g}_x)>\alpha$, for every $g\in \mathcal{U}$ and every $x\in\sad(g)$
whose orbit under $g$ is contained in $U$.
On the other hand, since $\de(f)$ has a dominated splitting we have $\de(f)\neq\emptyset$, and
so, there are sequence $f_k\to f$ and $x_k\in\sad_d(f_k)$ such that $\angle(x_k,f_k)\to 0$.
It follows from the definition of $\de(f)$ that the Hausdorff limit of the sequence formed by the orbit $\gamma_k$ of $x_k$ under $f_k$
is contained in $\de(f)$. Therefore, we can choose $k$ large so that $f_k\in \mathcal{U}$, $\gamma_k\subset U$
and $\angle(x_k,f_k)\leq\alpha$. This contradicts the choices of $\mathcal{U}$, $U$, $\alpha$ and the proof follows.
\end{proof}

\subsection{Domination outside degenerated points}
We have proved in Proposition \ref{arrasa} that there is no domination in the set of degenerated points.
Now we state the domination outside the degenerated points. We use this to prove a dissipative version of the main result in \cite{PS}.

The following result will be proved in Section \ref{sec-final}.

\begin{proposition}
\label{ladilla2}
If $f$ is a surface diffeomorphism, then
$\sad^*_d(f)\setminus \de(f)$ has a dominated splitting. 
\end{proposition}

Using it we can prove the following result.

\begin{theorem}
\label{peo}
There is an open and dense subset of surface diffeomorphisms $\mathcal{Q}$
such that if $f\in \mathcal{Q}$ and $\de(f)=\emptyset$, then $\sad_d^*(f)$ is hyperbolic. 
\end{theorem}

\begin{proof}
Define
$
\mathcal{U}=\{g\in \diff^1(M):\de(g)=\emptyset\}.
$
This is clearly an open subset of surface diffeomorphisms.

By the Kupka-Smale theorem \cite{hk} we can find a dense subset
$\mathcal{D}$ of $ \mathcal{U}$ formed by $C^2$ Kupka-Smale surface diffeomorphisms.
Furthermore, we can assume that every $g\in \mathcal{D}$
has neither normally contracting nor normally expanding
irrational circles (see \cite{PS} for the corresponding definition).

Let us prove that $\sad_d^*(g)$ hyperbolic for every $g\in\mathcal{D}$.
Take any $g\in \mathcal{D}$. 
Then, $\de(g)=\emptyset$ and so $\sad_d^*(g)$ has a dominated splitting by Proposition \ref{ladilla2}.
On the other hand, it is clear from the definition that every periodic point of $g$ in
$\sad_d^*(g)$ is saddle. Then, Theorem B in \cite{PS} implies that
$\sad_d^*(g)$ is the union of a hyperbolic set and normally contracting irrational circles. Since there are no such circles in $\mathcal{D}$
we are done.

We claim that for every $g\in\mathcal{D}$ there is a neighborhood $\mathcal{V}_g\subset \mathcal{U}$ of it
such that $\sad_d^*(h)$ is hyperbolic, $\forall h\in\mathcal{V}_g$.

Indeed, fix $g\in \mathcal{D}$. Since $\sad_d^*(g)$ is hyperbolic we can choose
a neighborhood $U_g$ of $\sad_d^*(g)$ and a neighborhood $\mathcal{V}_g$ of $g$ such that any compact invariant set of
any $h\in \mathcal{V}_g$ is hyperbolic (this is a well-known property of the hyperbolic sets, see \cite{hk}).
Applying Lemma \ref{l5} we can assume that $\sad_d^*(h)\subset U$, for every $h\in \mathcal{V}_g$,
proving the claim.

Define $\mathcal{Q}'$ as the union of the neighborhoods $\mathcal{V}_g$ as $g$ runs over $\mathcal{D}$.
Clearly $\mathcal{Q}'$ is open and dense in $\mathcal{U}$.
Define
$$
\mathcal{Q}=\mathcal{Q}'\cup\inte(\diff^1(M)\setminus\mathcal{Q}'),
$$
where $\inte(\cdot)$ denotes the interior operation.
Clearly $\mathcal{Q}$ is an open and dense subset of surface diffeomorphisms.

Now take $f\in \mathcal{Q}$ with $\de(f)=\emptyset$. Then,
$f\in \mathcal{U}$.
Since $f\in\mathcal{Q}$ we have either $f\in \mathcal{Q}'$ or
$f\in \inte(\diff^1(M)\setminus\mathcal{Q}')$.
If $f\in \inte(\diff^1(M)\setminus\mathcal{Q}')$,
we could arrange a neighborhood $\mathcal{V}$ of $f$ with $\mathcal{V}\cap\mathcal{Q}'=\emptyset$.
Since $f\in \mathcal{Q}$ which is open, we can assume
$\mathcal{V}\subset\mathcal{Q}$ contradicting the denseness of $\mathcal{Q}'$ in $\mathcal{Q}$.
Then, $f\in\mathcal{Q}'$ and we so $f\in\mathcal{V}_g$ for some $g\in \mathcal{D}$.
It follows from the claim that $\sad_d^*(f)$ is hyperbolic and we are done.
\end{proof}

\subsection{Degenerated points and approximation by dissipative tangencies}
In this subsection we establish the last of the aforementioned basic properties of the degenerated
points, that is, approximation by dissipative tangencies.

For this we will need the auxiliary definition below and its subsequently lemma.

\begin{defi}
\label{pretang}
A {\em dissipative homoclinic tangency} of a surface diffeomorphism
is a homoclinic tangency associated to a dissipative saddle.
As in p. 1446 of \cite{w}, we say that a point $x$ is a {\em dissipative prehomoclinic tangency} if it can be perturbed to become
a dissipative homoclinic tangency, i.e.,
if there are sequences $f_k\to f$ and $x_k\to x$ such that
$x_k$ is a dissipative homoclinic tangency
of $f_k$, $\forall k\in\mathbb{N}$.
We denote by $\ta^*(f)$ the set of dissipative prehomoclinic tangencies of $f$.
\end{defi}

Recall the residual subset $\mathcal{R}_0$ in Lemma \ref{lema0}.

\begin{lemma}
\label{aux}
If $f\in \mathcal{R}_0$, then
$\ta^*(f)\subset\cl(\sink(f))$. If, additionally, $f$ is orientation-preserving, then $\ta^*(f)\subset\cl(\sink_\mathbb{C}(f))$.
\end{lemma}

\begin{proof}
Fix $f\in\mathcal{R}_0$ and $x\in \ta^*(f)$.
Suppose by absurd that $x\notin\cl(\sink(f))$.
Then, by Lemma \ref{lema0},
we can select neighborhoods $U_x$, $\mathcal{U}_x(f)$ and $U$ of $x$, $f$ and $\cl(\sink(f))$
respectively satisfying (\ref{pp}).
As $x\in\ta^*(f)$ we
obtain $g\in\mathcal{U}_x(f)$ 
having a dissipative homoclinic tangency in $ U_x$.
Unfolding this tangency in the standard way (e.g. \cite{pt})
we obtain a surface diffeomorphism $h$ close to $g$ (and thus within $\mathcal{U}_x(f)$)
satisfying
$\sink(h)\cap U_x\neq\emptyset$.
As this contradicts (\ref{pp})
we obtain the result.

In the orientation-preserving case we just replace \cite{pt} by \cite{m0}
in order to obtain $h\in \mathcal{U}_x(f)$ with $\sink_\mathbb{C}(h)\cap U_x\neq\emptyset$.
\end{proof}

Unfolding tangencies as in the previous proof we can see that every surface diffeomorphism $f$ satisfies
$\ta^*(f)\subset\de(f)$.

The following property
of degenerated points is precisely that
the converse inclusion holds for Kupka-Smale surface diffeomorphisms.
Its proof will be given in Section \ref{sec-final}.

\begin{proposition}
\label{ladi1}
If $f$ is a Kupka-Smale surface diffeomorphism, then
$\de(f)=\ta^*(f)$.
\end{proposition}

As a first application we obtain
the following localization of the degenerated points. Recall again the residual subset $\mathcal{R}_0$ from Lemma \ref{lema0}.

\begin{corollary}
\label{ladilla1}
If $f\in \mathcal{R}_0$, then
$\de(f)\subset \cl(\sink(f))\setminus\sink(f)$.
If, additionally, $f$ is orientation-preserving, then $\de(f)\subset\cl(\sink_\mathbb{C}(f))\setminus\sink_\mathbb{C}(f)$.
\end{corollary}

\begin{proof}
Since every $f\in\mathcal{R}_0$ is Kupka-Smale, we obtain from
Lemma \ref{aux} and Proposition \ref{ladi1} that $\de(f)\subset\cl(\sink(f))$ (or $\cl(\sink_\mathbb{C}(f))$ in the orientation-preserving case).
But $\de(f)\subset \sad_d^*(f)$ and $\sad_d^*(f)\cap \sink(f)=\emptyset$, so $\de(f)\cap \sink(f)=\emptyset$
thus $\de(f)\subset \cl(\sink(f))\setminus\sink(f)$ (resp. $\cl(\sink_\mathbb{C}(f))\setminus\sink_\mathbb{C}(f)$).
\end{proof}

\subsection{Degenerated points and essential hyperbolicity}
In 1988 Araujo proved in his thesis \cite{A} under the guidance of Ma\~n\'e
that a $C^1$ generic surface diffeomorphism either has infinitely many sinks or is essentially hyperbolic
(more than this, he proved that the union of the basins of the hyperbolic attractors has full
Lebesgue measure).
The only available proofs of this result (both in Portuguese) are the original one \cite{A} and
the B. Santiago's dissertation under the guidance of the first author \cite{s}.

In this subsection we shall combine the previous results to obtain some equivalences for essentially hyperbolicity
(including Theorem \ref{hyp->ess}). A proof
of the equivalence (1) $\Longleftrightarrow$ (4) below (but with different arguments)
was sketched by R. Potrie in his draft note \cite{po}.

\begin{theorem}
\label{ara-theo}
There is a residual subset $\mathcal{R}_8$ of surface diffeomorphisms $f$ where
the following properties are equivalent:
\begin{enumerate}
\item $\sink(f)$ is finite;
\item Every dissipative homoclinic class of $f$ is hyperbolic;
\item $\de(f)=\emptyset$;
\item $f$ is essentially hyperbolic.
\end{enumerate}
\end{theorem}

\begin{proof}
First we shall construct the residual subset $\mathcal{R}_8$.

Let $\mathcal{R}_0$, $\mathcal{R}_2$, $\mathcal{R}_6$
and $\mathcal{Q}$ be the residual subsets in lemmas \ref{lema0}, \ref{neutral},
\ref{move-attractor} and the open subset in Theorem \ref{peo} respectively.

It follows from \cite{m1} (as amended in \cite{bc}) that
there is a residual subset $\mathcal{E}$ of surface diffeomorphisms $f$
exhibiting a residual subset $R_f\subset M$ such that
$\omega(x)$ is a quasi-attracting set, $\forall x\in R_f$.

Define $S: \diff^1(M)\to 2_c^M$ by $S(f)=\cl(\sad_d(f))$.
This map is clearly lower-semicontinuous, and so, upper-semicontinuous
in a residual subset of surface diffeomorphisms $\mathcal{K}$.

Observe that $\sad^*_d(f)=\cl(\sad_d(f))$ for every $f\in \mathcal{K}$.
Indeed, choose sequences $f_k\to f$ and $x_k\in\sad_d(f_k)$ such that $x_k\to x$ for some point $x$.
Fix a compact neighborhood $U$ of $\cl(\sad_d(f))$.
Since $S$ is upper-semicontinuous at $f$ and $f_k\to f$, we can select $k_0\in\mathbb{N}$ such that
$x_k\in U$ for $k\geq k_0$. Since $x_k\to x$ and $U$ is compact we conclude that $x\in U$.
As $U$ is an arbitrary neighborhood of $\cl(\sad_d(f))$, we get $x\in \cl(\sad_d(f))$ proving the assertion.

Define
$$
\mathcal{R}_8=\mathcal{R}_0\cap \mathcal{R}_2\cap \mathcal{R}_6\cap\mathcal{Q}\cap\mathcal{E}\cap\mathcal{K}.
$$
Then, $\mathcal{R}_8$ is also a residual subset of surface diffeomorphisms.

Now, we shall prove that the aforementioned properties are equivalent in $\mathcal{R}_8$.

\vspace{5pt}

Take $f\in\mathcal{R}_8$ satisfying (1).

Then, $\cl(\sink(f))\setminus\sink(f)=\emptyset$.

Since $f\in\mathcal{R}_0$, Corollary \ref{ladilla1} implies $\de(f)=\emptyset$.

Since $f\in\mathcal{Q}$, Theorem \ref{peo} implies that $\sad_d^*(f)$ is hyperbolic.

As $\cl(\sad_d(f))\subset \sad_d^*(f)$, we have that $\cl(\sad_d(f))$ is hyperbolic too.

But every dissipative homoclinic class belongs to $\cl(\sad_d(f))$ (by the Birkhoff-Smale theorem \cite{hk}).

Therefore, all such classes are hyperbolic.

Then, $f$ satisfies (2).

\vspace{5pt}

Now, take $f\in\mathcal{R}_8$ satisfying (2).

Using $f\in\mathcal{R}_2$, Lemma \ref{cm},
and the obvious fact that every dissipative saddle belongs to a dissipative homoclinic class,
we get
$\sad_d(f)\cap\cl(\sink(f))=\emptyset$.

Since $f\in \mathcal{R}_0$, Corollary \ref{ladilla1} implies
$\de(f)\subset \cl(\sink(f))$.

Therefore,
$\sad_d(f)\cap\de(f)=\emptyset.$

But clearly
$\sad_d(f)\subset \sad^*_d(f)$
so
$\sad_d(f)\subset \sad^*_d(f)\setminus \de(f).$

Moreover,
$\sad^*_d(f)\setminus \de(f)$
has a dominated splitting by Proposition \ref{ladilla2}.

Extending this last splitting to the closure in the standard way
\cite{bdv}, we obtain one for $\cl(\sad_d(f))$.

But $f\in\mathcal{K}$ so $\sad_d^*(f)=\cl(\sad_d(f))$ thus $\sad^*_d(f)$ has a dominated splitting too.

As $f\in \mathcal{R}_0$, we have that $f$ is Kupka-Smale and then $\cl(\sink(f))\setminus \sink(f)\subset \sad_d^*(f)$ by Lemma \ref{l9}.

Since $\sad_d^*(f)$ has a dominated splitting, we have that $\cl(\sink(f))\setminus \sink(f)$ also does.

But $f\in \mathcal{R}_0$ so $\de(f)\subset \cl(\sink(f))\setminus\sink(f)$ by Corollary \ref{ladilla1}.

Therefore, $\de(f)=\emptyset$ by Proposition \ref{arrasa}.

Thus, $f$ satisfies (3).

\vspace{5pt}

Next take $f\in \mathcal{R}_8$ satisfying (3).

Since $f\in \mathcal{Q}$, we conclude from Theorem \ref{peo} that
$\sad_d^*(f)$ is hyperbolic.

As $\cl(\sad_d(f))\subset \sad_d^*(f)$, we conclude the same for $\cl(\sad_d(f))$.

Since the periodic points are dense in $\cl(\sad_d(f))$,
we can apply the Smale's spectral theorem \cite{hk} to obtain
a decomposition
$$
\cl(\sad_d(f))=H_1\cup\cdots\cup H_r
$$
into
finitely many disjoint homoclinic classes $H_1,\cdots, H_r$.

Denote by $A_1,\cdots, A_l$ the elements of $\{H_1\cdots, H_r\}$ which are hyperbolic attractors, and by
$s_1,\cdots, s_i$ the orbits of the sinks of $f$.

Since $f\in \mathcal{E}$, we have that the aforementioned residual subset $R_f$ is well-defined.

If $x\in R_f$ we have that $\omega(x)$ is a quasi-attracting set.

Therefore,
$\omega(x)\cap(\cl(\sad_d(f))\cup\sink(f))$ by Lemma \ref{move-attractor} since
$f\in\mathcal{R}_6$.

If $\omega(x)\cap \sink(f)\neq\emptyset$, then $\omega(x)=s_j$ for some
$1\leq j\leq i$.

Otherwise, $\omega(x)\cap \cl(\sad_d(f))\neq\emptyset$ and so $\omega(x)\cap H_k\neq\emptyset$ for some
$1\leq k\leq r$.

As $f\in \mathcal{R}_2$ we have that $H_k$ is neutral, and so,
$\omega(x)\subset H_k$ by Lemma \ref{neutral}.

Therefore, $\omega(x)$ is a hyperbolic quasi-attracting set, and, then, a hyperbolic attractor.

From this we conclude that $\omega(x)=A_j$ for some $1\leq j\leq l$.

We conclude that $R_f$ is contained in the union of the basins
of the collection of hyperbolic attractors $\{A_1,\cdots, A_l,s_1,\cdots, s_i\}$.

Since $R_f$ is residual (and so dense) in $M$ we conclude
that the union of the basins of this collection is dense.

Since such an union is clearly open, we obtain that $f$ is essentially hyperbolic.

Therefore, $f$ satisfies (4).
Since (4) obviously implies (1), we are done.
\end{proof}

\section{Proofs}
\label{sec4}

\begin{proof}[Proof of Theorem \ref{attractor}]
Let $\mathcal{R}_0$ and $\mathcal{R}_7$ be the residual subsets in Lemma
\ref{lema0} and Corollary \ref{prove-attractor}.
By Theorem 2 p. 133 in \cite{abcd}
there is a residual subset $\mathcal{J}$ of surface diffeomorphisms $f$ for which
every attractor with a dominated splitting is hyperbolic.
Define,
$$
\mathcal{R}=\mathcal{R}_0\cap\mathcal{R}_7\cap\mathcal{J}.
$$
Then, $\mathcal{R}$ is a residual subset of surface diffeomorphisms.

Suppose by contradiction that there is $f\in\mathcal{R}$ exhibiting
an infinite sequence of non-trivial
attractors $A_k$, $k\in\mathbb{N}$.
Denote by $A$ the union of these attractors.
Since $f\in\mathcal{R}_7$,  Corollary \ref{prove-attractor} and the obvious fact that
$\cl(\sad_d(f))\subset \sad_d^*(f)$ imply
$A\subset \sad_d^*(f)\setminus \cl(\sink(f))$.
Then,
$A\subset \sad_d^*(f)\setminus \de(f)$ by Corollary \ref{ladilla2}
since $f\in\mathcal{R}_0$.
As
$\sad^*_d(f)\setminus \de(f)$
has a dominated splitting by Proposition \ref{ladilla2}, we have that
$A$ has a dominated splitting too.
Extending this last splitting to the closure in the standard way
\cite{bdv} we obtain that $\Lambda=\cl(A)$
has a dominated splitting.

Since $f\in \mathcal{R}_0$ we have that $f$ is Kupka-Smale.

Since $\Lambda$ has a dominated splitting, we have that each attractor $A_k$ has a dominated splitting.
Since $f\in\mathcal{J}$, we have that each $A_k$ is a
non-trivial hyperbolic attractor.

Since $f\in \mathcal{R}_7$, we have from Corollary \ref{prove-attractor}
that there is a sequence $x_k\in\sad_d(f)$ such that $A_k=H_f(x_k)$ for every $k\in\mathbb{N}$. 

Since the sequence $A_k$ is infinite, we obtain a contradiction by Theorem \ref{finiteness-attractors} applied to $\Lambda$.
This ends the proof.
\end{proof}

\begin{proof}[Proof of Theorem \ref{thAA'}]
For every surface diffeomorphism $f$ we set
$$
H(f)=\bigcup\{H:H\mbox{ is a hyperbolic dissipative homoclinic class of }f\}.
$$
Define the maps $S_+,S_-: \diff^1(M)\to 2^M_c$ by
$$
S_\pm(f)=\cl(H(f^{\pm 1})).
$$
These maps are clearly lower-semicontinuous, and so,
upper-semicontinuous in a residual subset $\mathcal{C}$.
Define $\mathcal{R}=\mathcal{C}\cap\mathcal{R}_0\cap \mathcal{R}_2$, where $\mathcal{R}_0$ and
$\mathcal{R}_2$
are the residual subsets in lemmas \ref{lema0} and \ref{cm} respectively.
Clearly $\mathcal{R}$ is a residual subset of surface diffeomorphisms.
As before, we can assume that $\mathcal{R}$ is symmetric.

Now, take $f\in \mathcal{R}$.
Then, $f,f^{-1}\in \mathcal{R}_0$ and so $\de(f)\subset \cl(\sink(f))$ and $\de(f^{-1})\subset\cl(\sou(f))$
by Corollary \ref{ladilla1}.
Since $f\in \mathcal{R}_2$, we also have
$H\cap (\cl(\sink(f))\cup\cl(\sou(f)))=\emptyset$ for every hyperbolic homoclinic class $H$
of $f$.

All together yield
$$
\de(f^\pm)\cap H_\pm(f)=\emptyset,
\quad\quad\forall f\in\mathcal{R}.
$$
But we also have $H_\pm(f)\subset\cl(\sad_d(f^{\pm 1}))$
by the Birkhoff-Smale theorem \cite{hk}.
As $\cl(\sad_d(f^{\pm 1}))\subset\sad_d^*(f^{\pm 1})$,
we conclude from Proposition \ref{ladilla2} that both
$S_+(f)$  and $S_-(f)$ have dominated splittings, $\forall f\in \mathcal{R}$.
Now, applying the upper-semicontinuity of $S_\pm$ at $\mathcal{C}$
and the persistence of dominated splittings \cite{bdv} we conclude
that for every $f\in \mathcal{R}$ there is a neighborhood
$\mathcal{V}_f$ of $f$ such that
both $S_+(h)$ and $S_-(h)$ have dominated splittings, for every $h\in\mathcal{V}_f$.

Define
$$
\mathcal{O}=\displaystyle\bigcup_{f\in\mathcal{R}}\mathcal{V}_f
$$
which is clearly open and dense in $\diff^1(M)$.
Define $\mathcal{D}$ as the set of $C^2$ Kupka-Smale diffeomorphisms $g$ in $\mathcal{O}$
for which every homoclinic class of is dissipative for either $g$ or $g^{-1}$.
It follows from the Kupka-Smale theorem \cite{hk} that $\mathcal{D}$ is dense in $\diff^1(M)$.

If $g\in\mathcal{D}$ then both $S_+(g)$ and $S_-(g)$ have dominated splittings.
As clearly every periodic point in these sets are saddles, and there are no
normally hyperbolic irrational circles in $S_+(g)\cup S_-(g)$, we conclude from Theorem B in \cite{PS}
that both $S_+(g)$ and $S_-(g)$ are hyperbolic sets.
Since the number of homoclinic classes in a hyperbolic set is finite,
we conclude that the number of hyperbolic homoclinic classes of $g$ which are dissipative for either $g$ or $g^{-1}$ is finite.
As every homoclinic class of $g\in\mathcal{D}$ is dissipative for either $g$ or $g^{-1}$, we conclude
that the number of hyperbolic homoclinic classes of $g$ is finite, $\forall g\in \mathcal{D}$.
Since $\mathcal{D}$ is dense in $\diff^1(M)$, we are done.
\end{proof}

\begin{proof}[Proof of Theorem \ref{hyp->ess}]
Apply Theorem \ref{ara-theo}.
\end{proof}

\begin{proof}[Proof of Theorem \ref{araujo}]
Let $\mathcal{R}_0$ and $\mathcal{R}_8$ be the residual subsets in Lemma \ref{lema0} and Theorem \ref{ara-theo} respectively.
Define $\mathcal{R}=\mathcal{R}_0\cap \mathcal{R}_8$. Then, $\mathcal{R}$ is a residual subset of surface diffeomorphisms.

Now, take $f\in\mathcal{R}$ orientation-preserving such that $\sink_\mathbb{C}(f)$ is finite.
Then, $\cl(\sink_\mathbb{C}(f))\setminus\sink_\mathbb{C}(f)=\emptyset$.
Since $f\in\mathcal{R}_0$ is orientation-preserving, Corollary \ref{ladilla1} implies $\de(f)=\emptyset$.
As $f\in\mathcal{R}_8$, Theorem \ref{ara-theo} implies that $f$ is essentially hyperbolic.
This ends the proof.
\end{proof}

\begin{proof}[Proof of Corollary \ref{c1}]
Let $\mathcal{R}_0$ and $\mathcal{R}_8$ be the residual subsets in Lemma \ref{lema0}
and Theorem \ref{ara-theo} respectively.
Define $\mathcal{R}=\mathcal{R}_0\cap \mathcal{R}_8$.
Then, $\mathcal{R}$ is a residual subset of surface diffeomorphisms.

Now, take $f\in \mathcal{R}$ and suppose by contradiction that
$\cl(\sink(f))\setminus \sink(f)$ has a dominated splitting. In particular, $\cl(\sink(f))\setminus \sink(f)\neq\emptyset$ and so
$\sink(f)$ is infinite.
Since $f\in\mathcal{R}_0$, Corollary \ref{ladilla1} implies $\de(f)\subset\cl(\sink(f))\setminus\sink(f)$ and so
$\de(f)=\emptyset$ by Proposition \ref{arrasa}.
As $f\in\mathcal{R}_8$, Theorem \ref{ara-theo} implies that $\sink(f)$ is finite, a contradiction.
This contradiction ends the proof.
\end{proof}

\begin{proof}[Proof of Corollary \ref{thA}]
Let $\mathcal{R}_{8}$ be the residual subset in Theorem \ref{ara-theo}.
As before, we can assume that $\mathcal{R}_8$ is symmetric.
By the Ma\~n\'e's dichotomy (Corollary II p.506 in \cite{M}) there is another residual subset $\mathcal{H}$ of surface diffeomorphisms $f$
which satisfy Axiom A if and only if both $\sink(f)$ and $\sou(f)$ are finite.
Define $\mathcal{R}=\mathcal{R}_8\cap \mathcal{H}$. Then, $\mathcal{R}$ is a residual subset of surface diffeomorphisms.

Now, take $f\in \mathcal{R}_8$ and suppose that every homoclinic class of $f$ is hyperbolic.
Then, every dissipative homoclinic class of  $f$ (resp. $f^{-1}$) is hyperbolic.
As $f^{\pm1}\in \mathcal{R}_8$, we conclude from Theorem \ref{ara-theo} that both $\sink(f)$ and $\sou(f)=\sink(f^{-1})$ are finite.
As $f\in \mathcal{H}$, we conclude that $f$ satisfies Axiom A and we are done.
\end{proof}

\section{Proof of propositions \ref{ladilla2} and \ref{ladi1}}
\label{sec-final}

\noindent
In both proofs we shall use the following notation.
Given a periodic point $x$ of a surface diffeomorphism $g$, we denote $\lambda(x,g)$
and $\sigma(x,g)$ the eigenvalues of $x$ with
$$
0<|\lambda(x,g)|\leq |\sigma(x,g)|.
$$

To prove Proposition \ref{ladilla2} we follow the proof of Proposition 3.7 in \cite{w} (which in turns followed Lemma 2.0.1 in \cite{PS}).
Since some parts are slightly different from those in \cite{w}, we include
the full details for the sake of completeness.

We start with a linear algebra assertion extracted from p. 967 of \cite{PS}.

\begin{lemma}
\label{extract}
For every $\alpha>0$ and every $\Delta>0$ there is $0<\beta_0<1$ such that
\begin{equation}
\label{rabano}
\|T-I\|\leq \Delta,
\end{equation}
for every linear map $T: V\to V$ of a two-dimensional inner product space $V$ satisfying $T/E=(1-\beta')I$ and $T/F=(1+\beta'')I$
for some pair of real numbers $\beta',\beta''$ with $|\beta'|,|\beta''|<\beta_0$ and some pair of subspaces $E,F\subset V$
with $\angle(E,F)>\alpha$.
\end{lemma}

\begin{proof}
Apply Lemma II.10 in \cite{M}.
\end{proof}

Next we state a
lemma whose proof
uses the argument in p. 967 of \cite{PS}.

\begin{lemma}
\label{sublemma}
For every surface diffeomorphism $f$ and every $x\notin\de(f)$ there are
$0<\lambda_x<1$ and neighborhoods $W_x$ of $x$ and $\mathcal{H}_x(f)$ of $f$ such that
$$
|\lambda(p,g)|<\lambda_x^{n_{p,g}},
\quad\quad\forall (p,g)\in(\sad_d(g)\cap W_x)\times \mathcal{H}_x(f).
$$
\end{lemma}

\begin{proof}
Since $x\notin \de(f)$ we can select $\alpha_x>0$ as well as neighborhoods $U_x$
and $\mathcal{W}_x(f)$ of $x$ and $f$ respectively such that
\begin{equation}
\label{ecopata}
\angle(p,g)>\alpha_x,\quad\quad\forall (p,g)\in(\sad_d(g)\cap U_x)\times \mathcal{W}_x(f).
\end{equation}

Put $W(f)=\mathcal{W}_x(f)$ in the Franks's Lemma to obtain the neighborhood $W_0(f)\subset W(f)$
of $f$ and $\epsilon>0$.
Set
$$
C=\sup\{\|Dh(x)\|:h\in W(f),x\in M\}.
$$
Put $\alpha=\alpha_x$ and $\Delta=\frac{\epsilon}{2C}$ in Lemma \ref{extract} to obtain $\beta_0>0$.

Now suppose that the conclusion of the lemma is not true.
Then, there are sequences $\delta_m\to 0^+$, $g_m\to f$ and $p_m\in\sad_d(g_m)$ with period $n_m=n_{p_m,g_m}$
such that
$(1-\delta_m)^{n_m}\leq |\lambda_m|$, for all $m\in\mathbb{N}$,
where $\lambda_m=\lambda(p_m,g_m)$.
Since $p\in \sad_d(g_m)$ we also have
$|\sigma_m|\leq (1-\delta_m)^{-m}$, for all $m\in \mathbb{N}$,
where $\sigma_m=\sigma(p_m,g_m)$.
Summarizing we have
$$
0<(1-\delta_m)^{n_m}\leq |\lambda_m|<1<|\sigma_m|\leq(1-\delta_m)^{-n_m},
\quad\quad\forall m\in\mathbb{N}.
$$
Choose $m$ such that
$$
g_m\in W_0(f)\quad\quad\mbox{ and }\quad\quad\frac{\delta_m}{1-\delta_m}<\beta_0.
$$
For simplicity we write
$\delta=\delta_m$, $g=g_m$ and $p=p_m$ and $n=n_m$.
Furthermore, we assume $0<\lambda<1<\sigma$.
Replacing above we obtain
$$
0<(1-\delta)^n\leq\lambda<1<\sigma\leq(1-\delta)^{-n},\quad g\in W_0(f)
\mbox{ and }\frac{\delta}{1-\delta}<\beta_0.
$$

Define the linear maps $T_i: T_{g^i(p)}M\to T_{g^i(p)}M$, $0\leq 1\leq n-1$ by
$$
T_i/E^{s,g}_{g^i(p)}=\lambda^{\frac{1}{n}}I
\quad\mbox{ and }\quad
T_i/E^{u,g}_{g^i(p)}=\sigma^{\frac{1}{n}}I.
$$
Notice that
$\lambda^{\frac{1}{n}}=1-\beta'$ and $\sigma^{\frac{1}{n}}=1+\beta''$ with
$$
\beta'=1-\lambda^{\frac{1}{n}}\leq \delta<\frac{\delta}{1-\delta}<\beta_0
$$
and
$$
\beta''=\sigma^{\frac{1}{n}}-1\leq (1-\delta)^{-1}-1=\frac{\delta}{1-\delta}<\beta_0.
$$
Moreover, (\ref{ecopata}) implies
$\angle(E^{s,g}_{g^i(p)},E^{u,g}_{g^i(p)})>\alpha_x$ for all $0\leq i\leq n-1$.
Since $\alpha=\alpha_x$ and $\Delta=\frac{\epsilon}{2C}$, (\ref{rabano}) in Lemma \ref{extract} implies
$$
\|T_i-I\|<\frac{\epsilon}{2C},
\quad\forall 0\leq i\leq n-1.
$$

Take two independent vectors $u_0,v_0\in T_pM$ with $\angle(u_0,v_0)\leq \alpha_x$.
Define the linear map $S: T_pM\to T_pM$ by
$$
S(u_0)=(1-\beta)u_0\quad\mbox{ and }\quad S(v_0)=(1+\beta)v_0,
$$
where $0<\beta<1$ is small enough to guarantee
$$
\|S-I\|\cdot\|T_0\|<\frac{\epsilon}{2C}.
$$
Next we define the linear maps $L_i:T_{g^i(p)}M\to T_{g^{i+1}(p)}M$, $0\leq i\leq n-1$ by
$L_i=T_{i+1}\circ Dg(g^i(p))$ (for $0\leq i\leq n-2$) and
$L_{n-1}=S\circ T_0\circ Dg(g^{n-1}(p))$.
It follows from these choices that
\begin{equation}
\label{lula}
\displaystyle\prod_{i=0}^{n-1}L_i=S.
\end{equation}
Furthermore,
$$
\|L_i-Dg(g^i(p))\|\leq\epsilon, \quad\quad\forall 0\leq i\leq n-1.
$$
Indeed,
for $0\leq i\leq n-2$ one has
$$
\|L_i-Dg(g^i(p))\|\leq \|T_{i+1}-I\|C\leq \frac{\epsilon}{2C} C< \epsilon,
$$
and for $i=n-1$,
$$
\|L_{n-1}-Dg(g^{n-1}(p))\|\leq \|S\circ T_0-I\|C\leq
$$
$$
(\|S-I\|\cdot\|T_0\|+\|T_0-I\|)C
\leq
\left(\frac{\epsilon}{2C}+\frac{\epsilon}{2C}\right)C=\epsilon.
$$
Since $g\in W_0(f)$, we can apply the Franks's Lemma to
$x_i=g^i(p)$, $0\leq i\leq \tau-1$, in order to obtain a
diffeomorphism $\tilde{g}\in W(f)$
such that
$\tilde{g}=g$ along the orbit of $p$ under $g$ (thus $p$ is a periodic point of $\tilde{g}$ with $n_{x,\tilde{g}}=n$)
and $D\tilde{g}(\tilde{g}^i(p))=L_i$ for $0\leq i\leq n-1$.
It follows that
$D\tilde{g}^{n_{p,\tilde{g}}}(p)=\displaystyle\prod_{i=0}^{n-1}L_i$.
Then, (\ref{lula}) implies
$$
D\tilde{g}^{n_{p,\tilde{g}}}(p)=S.
$$
From this and the definition of $S$ we obtain
$E^{s,\tilde{g}}_p=<u_0>$, $E^{u,\tilde{g}}_p=<v_0>$, $\lambda(p,\tilde{g})=1-\beta$ and $\sigma(p,\tilde{g})=1+\beta$.
Therefore,
$$
p\in\sad(\tilde{g}),\quad
|\det(D\tilde{g}^{n_{p,\tilde{g}}}(p))|=1-\beta^2<1\quad
\mbox{ and }\quad
\angle(E^{s,\tilde{g}}_p,E^{u,\tilde{g}}_p)=\angle(u_0,v_0).
$$
We have then find $\tilde{g}\in \mathcal{W}_x(f)$ and $p\in \sad_d(\tilde{g})\cap U_x$
such that
$$
\angle(p,\tilde{g})\leq \angle(u_0,v_0)\leq \alpha_x
$$
in contradiction with (\ref{ecopata}).
This ends the proof.
\end{proof}

\begin{proof}[Proof of Proposition \ref{ladilla2}]
To prove that $\sad^*_d(f)\setminus\de(f)$ has a dominated splitting we argue as
in the last paragraph of p. 1455 in \cite{w}.

More precisely, we shall prove that
there is $J_0\in\mathbb{N}$ such that
for any sequences $h_k\to f$ and $p_k\in\sad_d(h_k)$ with $p_k\to x\notin\de(f)$ there is $k_0\in\mathbb{N}$ such that
if $k\geq k_0$ then
$$
\frac{\|Dh^j_k(p_k)/E^{s,h_k}_{p_k}\|}{\|Dh^j_k(p_k)/E^{u,h_k}_{p_k}\|}\leq\frac{1}{2},
\mbox{ for some } 0\leq j\leq J_0.
$$

Assume by contradiction that this is not true.

Then, there are sequences $h_k\to f$, $p_k\in \sad_d(h_k)$, $j_k\to\infty$ and $x\notin\de(f)$ 
such that
$$
p_k\to x\quad\mbox{ and }\quad\frac{\|Dh^j_k(p_k)/E^{s,h_k}_{p_k}\|}{\|Dh^j_k(p_k)/E^{u,h_k}_{p_k}\|}>\frac{1}{2},
\quad\quad\forall k\geq1\mbox{ and }0<j\leq j_k.
$$

Since $x\not\in \de(f)$, we can apply Lemma \ref{sublemma} together with the definition of $\de(f)$
to select $\alpha_x>0$ as well as neighborhoods $U_x$
and $\mathcal{W}_x(f)$ of $x$ and $f$ respectively such that
\begin{equation}
\label{eco}
|\lambda(p,g)|<\lambda^{n_{p,g}}\mbox{ and }\angle(p,g)>\alpha_x,\quad\quad\forall (p,g)\in(\sad_d(g)\cap U_x)\times \mathcal{W}_x(f).
\end{equation}

Using the first inequality in (\ref{eco}) (e.g. p. 1456 in \cite{w}) we can prove that the associated period sequence
$n_{p_k,h_k}$ is unbounded, so, we can assume
$n_{p_k,h_k}\to\infty$.

Put $W(f)=\mathcal{W}_x(f)$ in the Franks's Lemma to obtain
the neighborhood $W_0(f)\subset W(f)$ of $f$ and $\epsilon>0$.
Set
$$
C=\sup\{\|Dh(x)\|:h\in W(f),x\in M\}.
$$

Choose $\epsilon_0>0$, $\epsilon_1>0$ and $m\in\mathbb{N}^+$ satisfying
\begin{equation}
\label{bala}
(2\epsilon_0+\epsilon_0^2)C\leq\epsilon,\quad
(1+\epsilon_1)\lambda<1,
\quad\quad\epsilon_1<\frac{\alpha_x}{1+\alpha_x}\epsilon_0
\end{equation}
and
\begin{equation}
\label{lapa}
\epsilon_1(1+\epsilon_1)^{m}\geq\frac{2}{\alpha_x}+4.
\end{equation}

Since $h_k\to f$, $p_k\to x$, $n_{p_k,h_k}\to\infty$ and $j_k\to\infty$, we can fix $k\in\mathbb{N}$ such that
$$
h_k\in W_0(f),\quad p_k\in U_x,\quad
n_{p_k,h_k}>2m
\quad\mbox{ and }\quad
\frac{\|Dh_k^{m}(p_k)/E^{s,h_k}_{p_k}\|}{\|Dh_k^{m}(p_k)/E^{u,h_k}_{p_k}\|}>\frac{1}{2}.
$$
Writing
$$
h=h_k,\quad p=p_k,\quad E^s_p=E^{s,h_k}_{p_k},\quad E^u_p=E^{u,h_k}_{p_k},\quad \tau=n_{p_k,h_k}
$$
and replacing above we obtain
\begin{equation}
\label{pepa}
h\in W_0(f),\quad p\in U_x,\quad
\tau>2m
\quad\mbox{ and }\quad
\frac{\|Dh^{m}(p)/E^{s}_{p}\|}{\|Dh^{m}(p)/E^{u}_{p}\|}>\frac{1}{2}.
\end{equation}
By the last inequality above we can fix unitary vectors $v\in E^u_p$ and $w\in E^s_p$ such that
\begin{equation}
\label{ll}
\frac{\|Dh^{m}(p)v\|}{\|Dh^{m}(p)w\|}<2.
\end{equation}

Define linear maps $P,S: T_pM\to T_pM$ by
$$
\left\{
\begin{array}{rcl}
P(w)=&w\\
P(v)=&v+\epsilon_1w
\end{array}
\right.
$$
and
$$
\left\{
\begin{array}{rcl}
S(w)=&w\\
S(v)=&v-\bigg(\epsilon_1(1+\epsilon_1)^{\tau-2m-1}\lambda(p,h)\sigma^{-1}(p,h)\bigg)w.
\end{array}
\right.
$$

It follows from (\ref{eco}), the last two inequalities in (\ref{bala}) and Lemma II.10 in \cite{M} that
$$
\|P-I\|\leq \left(\frac{1+\alpha_x}{\alpha_x}\right)\epsilon_1<\epsilon_0
$$
and
$$
\|S-I\|\leq \left(\frac{1+\alpha_x}{\alpha_x}\right)\bigg(\epsilon_1(1+\epsilon_1)^{\tau-2m-1}\lambda(p,h)\sigma^{-1}(p,h)\bigg)<
$$
$$
\left(\frac{1+\alpha_x}{\alpha_x}\right)\epsilon_1
<\epsilon_0.
$$
Therefore,
\begin{equation}
\label{diablodado}
\|P-I\|<\epsilon_0
\quad\quad\mbox{ and }\quad\quad
\|S-I\|<\epsilon_0.
\end{equation}

Take linear maps $T_j: T_{h^j(p)}M\to T_{h^j(p)}M$, $0\leq j\leq \tau-1$,
$$
T_j/E^s_{h^j(p)} = \left\{
\begin{array}{rcl}
(1+\epsilon_1)I,& \mbox{if} & 0\leq j\leq m\\
(1+\epsilon_1)^{-1}I, & \mbox{if} & m+1\leq j\leq \tau-1
\end{array}
\right.
$$
and
$$
T_j/E^u_{h^j(p)}=I,
\quad\quad\forall 0\leq j\leq \tau-1,
$$

Applying (\ref{bala}) and Lemma II.10 in \cite{M} we obtain
$$
\|T_j-I\|\leq \frac{1+\alpha_x}{\alpha_x}\epsilon_1<\epsilon_0, \quad\quad\forall 0\leq j\leq m,
$$
and
$$
\|T_j-I\|\leq \frac{1+\alpha_x}{\alpha_x}\frac{\epsilon_1}{1+\epsilon_1}<\epsilon_0,
\quad\quad \forall m+1\leq j\leq \tau-1.
$$
Therefore,
\begin{equation}
\label{extra}
\|T_j-I\|\leq \epsilon_0, \quad\quad\forall 0\leq j\leq \tau-1.
\quad\quad\forall 0\leq j\leq \tau-1.
\end{equation}

Define the linear maps $L_j:T_{h^j(p)}M\to T_{h^{j+1}(p)}M$, $0\leq j\leq \tau-1$, by
$$
L_j = \left\{
\begin{array}{rcl}
T_1\circ Dh(p)\circ P,& \mbox{if} & j=0\\
T_{j+1}\circ Dh(h^j(p)), & \mbox{if} & 1\leq j\leq \tau-2\\
S\circ T_0\circ Dh(h^{\tau-1}(p)), & \mbox{if} & j=\tau-1.
\end{array}
\right.
$$

We can use the first inequality in (\ref{bala}), (\ref{diablodado}) and (\ref{extra}) as in \cite{PS} to prove
$$
\|L_j-Dh(h^j(p))\|\leq \epsilon,
\quad\quad\forall 0\leq j\leq \tau-1.
$$
Indeed, for $j=0$ one has
$$
\|L_0-Dh(p)\|
\leq \|T_1\circ Dh(p)\circ P-T_1\circ Dh(p)\|+\|T_1\circ Dh(p)-Dh(p)\|\leq
$$
$$
(1+\epsilon_0)C\|P-I\|+\|T_1-I\|C\leq
(2\epsilon_0+\epsilon_0^2)C\leq \epsilon,
$$
for $1\leq j\leq \tau-2$ one has
$$
\|L_j-Dh(h^j(p))\|=\|T_{j+1}\circ Dh(h^j(p))-Dh(h^j(p))\|\leq \|T_j-I\|C\leq \epsilon_0C\leq\epsilon,
$$
and for $j=\tau-1$ one has
$$
\|L_{\tau-1}-Dh(h^{\tau-1}(p))\|\leq(\|S-I\|\cdot\|T_0\|+\|T_0-I\|)C\leq
$$
$$
(\epsilon_0(1+\epsilon_0)+\epsilon_0)C=(2\epsilon_0+\epsilon_0^2)C\leq\epsilon.
$$

Since $h\in W_0(f)$, we can apply the Franks' s Lemma to
$x_i=h^i(p)$, $0\leq i\leq \tau-1$, in order to obtain a
diffeomorphism $g$ with
\begin{equation}
\label{laleli}
 g\in \mathcal{W}_x(f)
\end{equation}
such that
$g=h$ along the orbit of $p$ under $h$ (thus $p$ is a periodic point of $g$ with $n_{x,g}=\tau$)
and $Dg(g^i(p))=L_i$ for $0\leq i\leq \tau-1$.
It follows that
$$
Dg^{n_{p,g}}(p)=\displaystyle\prod_{j=0}^{\tau-1}L_j.
$$
Direct computations together with the definitions of $P$ and $S$ show
$$
Dg^{n_{p,g}}(p)w=(1+\epsilon_1)^{-\tau+2m+1}\lambda(p,h)w
$$
and
\begin{eqnarray*}
Dg^{n_{p,g}}(p)v & = & S\bigg(\sigma(p,h) v+\epsilon_1(1+\epsilon_1)^{\tau-2m-1}\lambda(p,h) w\bigg) \\
& = & \sigma(p,h)S\bigg(v+\epsilon_1(1+\epsilon_1)^{\tau-2m-1}\lambda(p,h)\sigma^{-1}(p,h)w\bigg)\\
& = & \sigma(p,h)\bigg(v-
\epsilon_1(1+\epsilon_1)^{\tau-2m-1}\lambda(p,h)\sigma^{-1}(p,h)w+\\
& &\epsilon_1(1+\epsilon_1)^{\tau-2m-1}\lambda(p,h)\sigma^{-1}(p,h)w
\bigg)\\
& = & \sigma(p,h)v.
\end{eqnarray*}
 
As $|(1+\epsilon_1)^{-\tau+2m+1}\lambda(p,h)|<1$ and $|\sigma(p,h)|>1$ we obtain
\[
E^{s,g}_p=E^s_p, \quad
\lambda(p,g)=(1+\epsilon_1)^{-\tau+2m+1}\lambda(p,h),\quad
E^{u,g}_p=E^u_p\mbox{ and }
\sigma(p,g)=\sigma(p,h).
\]
In particular,
$$
|\det(Dg^{n_{p,g}}(p))|=(1+\epsilon_1)^{-\tau+2m+1}|\lambda(p,h)\sigma(p,h)|\leq
|\det(Dh^{n_p}(p))|<1
$$
proving
\begin{equation}
 \label{pape}
p\in \sad_d(g).
\end{equation}

Next we proceed as in p. 971 of \cite{PS}.

Since $w\in E^{s,g}_p$ and $v\in E^{u,g}_p$ we obtain,
$u_1=Dg^m(p)(v)\in E^{s,g}_{g^m(p)}$ and $u_2=Dg^m(p)(w)\in E^{u,g}_{g^m(p)}$.
But
$$
u_1=Dh^m(p)(v)+\epsilon_1(1+\epsilon_1)^mDh^m(p)(w)\quad\mbox{ and }\quad
u_2=(1+\epsilon_1)^mDh^m(p)(w).
$$
As $\epsilon_1u_2\in E^{s,g}_{g^m(p)}$, Lemma II.10 in \cite{M} yields
$$
\|u_1-\epsilon_1u_2\|\geq\frac{\beta}{1+\beta}\|u_1\|,
\quad\mbox{ where }\quad
\beta=\angle(E^{u,g}_{g^{m}(p)},E^{s,g}_{g^{m}(p)}).
$$
Consequently,
$$
\|Dh^m(p)(v)\|=\|u_1-\epsilon_1u_2\|\geq\frac{\beta}{1+\beta}\|u_1\|
\geq
$$
$$
\frac{\beta}{1+\beta}\left(\epsilon_1(1+\epsilon_1)^m\|Dh^m(p)(w)\|-\|Dh^m(p)(v)\|\right)
\overset{(\ref{ll})}{\geq}
$$
$$
\frac{\beta}{1+\beta}\left(
\frac{\epsilon_1(1+\epsilon_1)^m}{2}-1
\right)\|Dh^m(p)v\|
$$
thus
$$
\frac{1+\beta}{\beta}\geq\frac{\epsilon_1(1+\epsilon_1)^m}{2}-1.
$$
As $\beta=\angle(E^{u,g}_{g^{m}(p)},E^{s,g}_{g^{m}(p)})$ we obtain
$$
\angle(E^{u,g}_{g^{m}(p)},E^{s,g}_{g^{m}(p)})\leq\frac{2}{\epsilon_1(1+\epsilon_1)^m-4}.
$$
Applying (\ref{lapa}) we arrive to $\angle(E^{u,g}_{g^{m}(p)},E^{s,g}_{g^{m}(p)})\leq\alpha_x$ so
$\angle(p,g)\leq\alpha_x$. But now (\ref{pepa}), (\ref{laleli}) and (\ref{pape}) contradict the second inequality in (\ref{eco})
and the proof follows.
\end{proof}

To prove Proposition \ref{ladi1} we will need the following numerical assertion extracted
from p. 975 of \cite{PS}.

\begin{lemma}
\label{numerical}
For every pair of positive numbers $\epsilon'$ and $\epsilon_1$, every pair of sequences of positive numbers $\gamma_n\to0$ and $\sigma_n\geq1$,
and every integer sequence $m_n\to\infty$
there is $n\in\mathbb{N}$ such that if $\tilde{\sigma}_n=(1+\delta_n)^{m_n}\sigma_n$ and $\delta_n=\frac{\gamma_n\epsilon'}{2}$, then
\begin{equation}
\label{pela}
\frac{\tilde{\sigma}_n-1}{\tilde{\sigma}_n+1}\frac{\epsilon_1}{2}>\gamma_n.
\end{equation}
\end{lemma}

\begin{proof}
The proof depends on whether the sequence $\tilde{\sigma}_n$ is bounded or not.
If $\tilde{\sigma}_n$ is not bounded, then such an $n$ exists because
$\gamma_n\to0$ and
$\frac{\tilde{\sigma}_n-1}{\tilde{\sigma}_n+1}\frac{\epsilon_1}{2}\to \frac{\epsilon_1}{2}$ (up to passing to a subsequence if necessary).
Then, we can suppose that $\tilde{\sigma}_n$ is bounded, i.e., there is $K>0$ such that
$$
\tilde{\sigma}_n\leq K,
\quad\quad\forall n\in\mathbb{N}.
$$
Since $m_n\to\infty$ we can select $n$ satisfying
\begin{equation}
\label{mn}
\frac{m_n\epsilon'\epsilon_1}{4(K+1)}>1.
\end{equation}
As
\begin{eqnarray*}
\frac{\tilde{\sigma}_n-1}{\tilde{\sigma}_n+1} & = &  \frac{(1+\delta_n)^{m_n}\sigma_n-1}{\tilde{\sigma}_n+1}
\geq
\frac{(1+m_n\delta_n)\sigma_n-1}{\tilde{\sigma}_n+1}\\
& = & \frac{\sigma_n-1}{\tilde{\sigma}_n+1}
+
\frac{m_n\delta_n\sigma_n}{\tilde{\sigma}_n+1}
\geq \frac{m_n\delta_n}{K+1}=\frac{m_n\epsilon'}{2(K+1)}\gamma_n
\end{eqnarray*}
we obtain
$$
\frac{\tilde{\sigma}_n-1}{\tilde{\sigma}_n+1}\frac{\epsilon_1}{2}>
\frac{m_n\epsilon'\epsilon_1}{4(K+1)}\gamma_n\overset{(\ref{mn})}{>}\gamma_n
$$
so $n$ satisfies (\ref{pela}).
\end{proof}

\begin{proof}[Proof of Proposition \ref{ladi1}]
It suffices to prove that for every $x\in \de(f)$, every neighborhood $U$ of $x$
and every neighborhood $\mathcal{U}$ of $f$
there is $g\in\mathcal{U}$
having a dissipative homoclinic tangency in $U$.
For this we proceed as in the proof of Lemma 2.2.2 in \cite{PS}.

Fix neighborhoods $\mathcal{U}_1(f)\subset \mathcal{U}_0(f)\subset \mathcal{U}$ of $f$ and $\epsilon_1>0$ such that every $g$
that is $\epsilon_1$-close to some $\tilde{g}\in \mathcal{U}_1(f)$ belongs to $\mathcal{U}_0(f)$.

Put $W(f)=\mathcal{U}_1(f)$ in the Franks's Lemma to obtain the neighborhood $\mathcal{U}_2(f)\subset \mathcal{U}_1(f)$ of $f$ and $\epsilon>0$.
Set
$$
C=\sup\{\|Dh\|:h\in\mathcal{W}\}\quad\mbox{ and fix }\quad
0<\epsilon'<\frac{\epsilon}{C}.
$$

Now, take $x\in\de(f)$.
Then, there are sequences $g_n\to f$ and $\hat{p}_n\in\sad_d(g_n)$
such that $\hat{p}_n\to x$ and $\gamma_n\to 0$,
where

\begin{equation}
\label{angle1}
\gamma_n=\angle(\hat{p}_n,g_n).
\end{equation}
For all $n\in\mathbb{N}$ we select $p_n$ in the orbit of $\hat{p}_n$ under $g_n$ such that
$\gamma_n=\angle(E^{s,g_n}_{p_n},E^{u,g_n}_{p_n})$.

Since $g_n\to f$ and $\hat{p}_n\to x$ we can assume
$$
g_n\in\mathcal{U}_2(f)\quad\mbox{ and }\quad\hat{p}_n\in U,
\quad\quad\forall n\in\mathbb{N}.
$$

We claim that there are $n\in \mathbb{N}$, and
a surface diffeomorphism $\tilde{g}\in W(f)$
such that
$\tilde{g}=g_n$ along the orbit of $p_n$ under $g_n$ (thus $p_n$ is a periodic point of $\tilde{g}$ with $n_{p_n,\tilde{g}}=n_{p_n,g_n}$)
such that
\begin{equation}
\label{claim}
p_n\in\sad_d(\tilde{g})\quad\quad\mbox{ and }\quad\quad
\angle(E^{s,\tilde{g}}_{p_n},E^{u,\tilde{g}}_{p_n})<
\frac{|\sigma(p_n,\tilde{g})|-1}{|\sigma(p_n,\tilde{g})|+1}\frac{\epsilon_1}{2}
\end{equation}
Indeed, if, for some $n\in\mathbb{N}$,
$$
\gamma_n<\frac{|\sigma(p_n,g_n)|-1}{|\sigma(p_n,g_n)|+1}\frac{\epsilon_1}{2},
$$
we just take $\tilde{g}=g_n$.

Therefore, we can assume
\begin{equation}
\label{papa}
\gamma_n\geq\frac{|\sigma(p_n,g_n)|-1}{|\sigma(p_n,g_n)|+1}\frac{\epsilon_1}{2}, \quad\quad\forall n\in\mathbb{N}.
\end{equation}

To simplify we write $m_n=n_{p_n,g_n}$.

Since $f$ is Kupka-Smale we can assume that $m_n\to\infty$
(otherwise, as $\gamma_n\to0$, the limit point of $p_n$ would be a non-hyperbolic periodic point because of (\ref{papa})).

Since $\gamma_n\to 0$, we can assume
$$
\gamma_n<1,\quad\quad\forall n\in\mathbb{N}.
$$
Putting
$\epsilon'$, $\epsilon_1$, $\gamma_n\to0$ and $m_n$ as above and $\sigma_n=|\sigma(p_n,g_n)|$ in
Lemma \ref{numerical}, we obtain $n$ satisfying (\ref{pela}) where
$$
\tilde{\sigma}_n=(1+\delta_n)^{m_n}\sigma_n\quad\quad\mbox{ and }
\quad\quad\delta_n=\frac{\gamma_n\epsilon'}{2}.
$$

With this $n$ in hand, we define the linear automorphisms
$T_i$ of $T_{g^i_n(p_n)}M$ by
$$
T_i/E^{s,g_n}_{p_n}=(1-\delta_n)I\quad \mbox{ and }\quad T_i/E^{u,g_n}_{p_n}=(1+\delta_n)I,
\quad\forall 0\leq i\leq m_n-1,
$$
where $I$ above is the identity.

On the other hand, applying Lemma II.10 in \cite{M}, the definition of $\delta_n$ and (\ref{angle1}) we obtain
$$
\|T_i-I\|\leq \frac{1}{2}\left(\gamma_n+\frac{\gamma_n}{\angle(E^{s,g_n}_{g^i_n(p_n)},E^{u,g_n}_{g^i_n(p_n)})}\right)\epsilon'
\leq\frac{1}{2}\cdot 2\cdot\epsilon'=\epsilon',
$$
i.e., $\|T_i-I\|\leq \epsilon'$ for all $0\leq i\leq m_n-1$.

We also define the linear maps
$L_i:T_{g^i_n(p_n)}M\to T_{g^{i+1}_n(p_n)}M$ by
$L_i=T_{i+1}\circ Dg_n(g^i_n(p_n))$ (for $0\leq i\leq m_n-2$) and
$L_{m_n-1}=T_0\circ Dg_n(g^{m_n-1}_n(p_n))$.

Since $\|T_i-I\|\leq\epsilon'$ one has
$$
\|L_i-Dg_n(g^i_n(p_n))\|\leq\epsilon,
\quad\quad\forall 0\leq i\leq m_n-1.
$$
Since $g_n\in \mathcal{U}_2(f)$, we can apply the Franks' s Lemma to
$x_i=g_n^i(p_n)$, $0\leq i\leq m_n-1$, in order to obtain a
diffeomorphism $\tilde{g}\in \mathcal{U}_1(f)$
such that
$\tilde{g}=g_n$ along the orbit of $p_n$ under $g_n$ (thus $p_n$ is a periodic point of $\tilde{g}$ with $n_{p_n,\tilde{g}}=m_n$)
and $D\tilde{g}(\tilde{g}^i(p_n))=L_i$ for $0\leq i\leq m_n-1$.
Consequently,
$$
D\tilde{g}^{n_{p_n,\tilde{g}}}(p_n)=\displaystyle\prod_{i=0}^{m_n-1}L_i.
$$
It follows from the choice of $L_i$ and the above identity that
the eigenvalues of $p_n$ as a periodic point of $\tilde{g}$ are
$(1-\delta_n)^{m_n}\lambda(p_n,g_n)$ (less than $1$ in modulus) and
$\tilde{\sigma}_n=(1+\delta_n)^{m_n}\sigma(p_n,g_n)$ (bigger than $1$ in modulus).
Therefore,
$$
|\det(D\tilde{g}^{n_{p_n,\tilde{g}}}(p_n))|=(1-\delta_n^2)^{m_n}|\lambda(p_n,g_n)\sigma(p_n,g_n)|
<1
$$
since $p_n\in\sad_d(g_n)$.
All together yield
$$
p_n\in\sad_d(\tilde{g}).
$$
It follows also that
$\angle(E^{s,\tilde{g}}_{p_n},E^{u,\tilde{g}}_{p_n})=\gamma_n$ so
(\ref{pela}) implies (\ref{claim}) and the claim follows.

Fixing $n$ and $\tilde{g}$ as in the claim,
we can use (\ref{claim}) and Lemma 2.2.2 in \cite{PS}  (as amended in Lemma 4.2 of \cite{w}) to find a surface diffeomorphism
$g$,
$\epsilon_1$ close to $\tilde{g}$, differing from $\tilde{g}$ only on a small ball near but disjoint from the orbit of
$p_n$ under $\tilde{g}$ (which is the same as that under $g_n$)
such that $g$ has a homoclinic tangency associated to $p_n$.
Propagating this tangency along the orbit of $p_n$ we obtain one close to
$\hat{p}_n$.
As $\hat{p}_n\in U$ we obtain that there is $g\in\mathcal{U}$
having a dissipative homoclinic tangency in $U$ and the proof follows.
\end{proof}

\end{document}